\documentclass[12pt,reqno]{amsart}
\usepackage{amssymb}
\usepackage{graphicx}
\usepackage{xcolor}
\usepackage{bm}

\usepackage[all]{xy}
\DeclareFontFamily{U}{mathb}{\hyphenchar\font45}
\DeclareFontShape{U}{mathb}{m}{n}{
      <5> <6> <7> <8> <9> <10> gen * mathb
      <10.95> mathb10 <12> <14.4> <17.28> <20.74> <24.88> mathb12
      }{}
\DeclareSymbolFont{mathb}{U}{mathb}{m}{n}
\DeclareMathSymbol{\righttoleftarrow}{3}{mathb}{"FD}

\theoremstyle{plain}
\newtheorem{prop}{Proposition}[section]
\newtheorem{theo}[prop]{Theorem}

\newtheorem{lemm}[prop]{Lemma}
\theoremstyle{remark}
\newtheorem{rema}[prop]{Remark}

\theoremstyle{definition}

\newtheorem{exam}[prop]{Example}
\newtheorem{ques}[prop]{Question}
\numberwithin{equation}{section}
\newcommand{\A}{{\mathbb A}}

\newcommand{\bP}{{\mathbb P}}

\newcommand{\G}{{\mathbb G}}

\newcommand{\R}{{\mathbb R}}
\newcommand{\Z}{{\mathbb Z}}

\newcommand{\cG}{{\mathcal G}}
\newcommand{\cO}{{\mathcal O}}
\newcommand{\cI}{{\mathcal I}}
\newcommand{\cJ}{{\mathcal J}}

\newcommand{\cL}{{\mathcal L}}

\newcommand{\cT}{{\mathcal T}}

\newcommand{\fA}{{\mathfrak A}}

\newcommand{\rH}{{\mathrm H}}

\newcommand{\GL}{{\mathrm{GL}}}

\newcommand{\rC}{{\mathrm C}}

\newcommand{\bG}{{\mathbb G}}

\newcommand{\bN}{{\mathbb N}}

\newcommand{\bR}{{\mathbb R}}
\newcommand{\bZ}{{\mathbb Z}}

\newcommand{\fS}{{\mathfrak S}}

\newcommand{\eqto}{\stackrel{\lower1.5pt\hbox{$\scriptstyle\sim\,$}}\to}
\newcommand{\eqdashto}{\stackrel{\lower1.5pt\hbox{$\scriptstyle\sim\,$}}\dashrightarrow}
\newcommand{\actsfromleft}{\mathrel{\reflectbox{$\righttoleftarrow$}}}
\newcommand{\actsfromright}{\righttoleftarrow}

\DeclareMathOperator{\Gal}{Gal}

\DeclareMathOperator{\Pic}{Pic}
\DeclareMathOperator{\Spec}{Spec}

\DeclareMathOperator{\Aut}{Aut}

\DeclareMathOperator{\Burn}{Burn}

\DeclareMathOperator{\Stab}{Stab}
\DeclareMathOperator{\Ind}{Ind}

\begin{document}
\title[Equivariant Burnside groups]{Equivariant Burnside groups and toric varieties}

\author{Andrew Kresch}
\address{
  Institut f\"ur Mathematik,
  Universit\"at Z\"urich,
  Winterthurerstrasse 190,
  CH-8057 Z\"urich, Switzerland
}
\email{andrew.kresch@math.uzh.ch}
\author{Yuri Tschinkel}
\address{
  Courant Institute,
  251 Mercer Street,
  New York, NY 10012, USA
}

\email{tschinkel@cims.nyu.edu}

\address{Simons Foundation\\
160 Fifth Avenue\\
New York, NY 10010\\
USA}

\date{December 9, 2021}

\begin{abstract}
We study $G$-equivariant birational geometry of toric varieties, where $G$ is a finite group.
\end{abstract}

\maketitle

\section{Introduction}
\label{sec.intro}

In this paper we continue our studies of actions of finite groups on algebraic varieties, up to equivariant birational equivalence. Our main tool is the new invariant introduced in \cite{BnG}: 
to an $n$-dimensional smooth projective variety $X$ with a regular action of a finite group $G$ we assigned a class
$$
[X\actsfromright G] \in \Burn_n(G),
$$
taking values in the {\em equivariant Burnside group}, which is defined by certain symbols and relations; this class is an equivariant birational invariant. In \cite{HKTsmall}, \cite{KT-struct} we studied various structural properties of this new invariant and provided first applications. In \cite{KT-vector} we presented an algorithm to compute the classes of linear actions and presented examples of nonbirational such actions. 
Here, we turn to algebraic tori.

Recall that an algebraic torus of dimension $n$ over a field $k$ is a linear algebraic group $T$ which is a $k$-form of $\mathbb G_m^n$. The absolute Galois group of $k$
acts on the geometric character group $M:=\mathfrak X^*(T_{\bar{k}})$ via a finite subgroup
\[ G\subset \Aut(M)\cong \mathrm{GL}_n(\Z). \]
A torus $T$ over $k$ is uniquely determined up to isomorphism by 
its splitting field, Galois over $k$, with Galois group $G$, and
this representation of $G$. 

Rationality properties of tori over nonclosed fields have been extensively studied, see, e.g., \cite{vosk}, \cite{EM}, \cite{Sansuc-CT}, \cite{hoshi}, \cite{lemire-tori}.
The relevant cohomological obstruction results from the Galois module $\Pic(X_{\bar k})$, for a smooth projective compactification $X$ of $T$. 
For any subgroup $G'\subseteq G$ the group
\begin{equation}
\label{eqn:group}
\rH^1(G', \Pic(X_{\bar{k}}))
\end{equation}
is independent of the choice of $X$; its nontriviality is
an obstruction to stable $k$-rationality of $T$. This is the only obstruction in dimensions $\le 3$; moreover, every stably rational torus in dimension $\le 3$ is rational  \cite{kun}. 
The Zariski problem for algebraic tori, i.e., the question of whether or not stably rational tori over $k$ are rational over $k$ is still open, in particular, for 4-dimensional tori identified in \cite[Prop. 4.15]{lemire-tori}, with $G$ a subgroup of $C_2\times \fA_5$ or 
$C_2\times \fS_4$.

Here, over an algebraically closed field $k$ of characteristic zero, we study the $G$-equivariant version of this question, for a finite group $G$:

\medskip

\noindent
{\bf Problem:}
{\it Is a given $G$-action on an algebraic torus linearizable, i.e., $G$-equivariantly birational to a linear action on projective space? }

\medskip

There are many similarities but also subtle distinctions between these points of view, highlighted, e.g., in \cite{HT-intersect}.
One of the similarities is that the cohomological obstruction \eqref{eqn:group} applies also as an obstruction to (stable) linearizability of the $G$-action, since for linear actions of $G$, the invariant \eqref{eqn:group} vanishes (see, e.g., \cite[Prop. 2.2]{BogPro} and references therein).  
On the other hand, all tori in dimension 2, over any field, are rational, while there is an action of $G:=C_2\times \fS_3$ on $\mathbb G_m^2$, which is not linearizable \cite{isk-s3}, but is stably linearizable \cite[Prop. 9.11]{lemire}; the corresponding class in $\Burn_2(G)$ is distinct from those of linear actions \cite[Section 7.6]{HKTsmall}.    

In the case of surfaces both points of view have received ample attention, going back to \cite{manin-ihes}, \cite{IskMin}, with further developments in \cite{DI}, \cite{BogPro}, \cite{prokhorovII}, and in many other papers. The main approach there is via the (equivariant) Minimal Model Program, i.e., classification of all birational models and (equivariant) birational transformations between those models.  
Much less is known about linearizability of $G$-actions 
in higher dimensions, in particular, for tori, see \cite{Ch-toric}.

Our main results in this paper are: 
\begin{itemize}
\item We give a recursive procedure (Theorem~\ref{thm.main}) to compute the class 
\begin{equation}
\label{eqn:class}
[X\actsfromright G] \in \Burn_n(G). 
\end{equation}
This uses the De Concini-Procesi formalism to construct a suitable equivariant birational model of the torus $T$.
\item We present an example of such a computation (Proposition~\ref{prop:class-dp6}). 
\item We discuss the relation between 
the class \eqref{eqn:class} and existing (stable) $G$-birational invariants, such as group cohomology \eqref{eqn:group}: there exist actions that can be distinguished by one invariant but not the other
(Proposition~\ref{prop:class-dp6} and Proposition~\ref{prop:kun}). 
%\item We find new examples of nonlinearizable actions in dimension 3. 
\end{itemize}

\

\noindent
{\bf Acknowledgments:} 
The first author was partially supported by the Swiss National Science Foundation. The second author was partially supported by NSF grant 2000099.
This paper is based upon work partially supported by the Swedish Research Council under grant no.~2016-06596 while the first author was in residence at the Institut Mittag-Leffler.

\section{Toric varieties: generalities}
\label{sec.generalities}

We work over an algebraically closed field $k$ of characteristic zero.

\subsection{Fans}
\label{sect:lf}
Let $T=\bG_m^n$ be an algebraic torus of dimension $n$, 
$$
M:=\mathfrak X^*(T), \quad \text{ respectively } \quad N:=\mathfrak X_*(T),
$$
the lattice of its algebraic characters, respectively, co-characters. 
A smooth projective toric variety 
$$
X=X_{\Sigma}
$$ 
of dimension $n$ is an equivariant compactification of $T$. It is uniquely determined the combinatorial structure of a fan 
$$
\Sigma=\{ \sigma\},
$$
a finite collection of strongly convex rational polyhedral cones $\sigma$ in
$N_{\R}:=N\otimes_{\Z}\R$ (see, e.g., \cite{fulton} for basic definitions concerning toric varieties). We let 
$$
\Sigma(d), \quad d=0,\dots, n,
$$
denote the collection of 
$d$-dimensional cones in $\Sigma$. 

The fan is subject to various conditions to insure smoothness and projectivity of $X$; see, e.g., \cite{batyrev}: 
\begin{itemize}
\item every cone $\sigma\in\Sigma$ is simplicial and is generated by a part of a basis of $N$, 
\item 
the union of cones is all of $N_{\bR}$, and 
\item 
$\Sigma$ admits a piecewise linear convex support function.
\end{itemize}
Such a fan  $\Sigma$ is called a smooth projective fan.

\subsection{Subtori and their closures} 
\label{sect:sc}
A primitive sublattice $N'\subseteq N$ gives rise to a {\em subtorus} $T'\subset T$, and an induced equivariant compactification $X'$ of $T'$. The corresponding fan $\Sigma'$ is the fan in $N'_{\bR}$ induced by $\Sigma$, that is, given by intersecting the cones of $\Sigma$ with $N'_{\bR}$. 
The interesting case for us is when $T'$ satisfies the following 

\medskip
\noindent
{\bf Property (E):} 
$\sigma\cap N'_{\bR}$ is a face of $\sigma$, for all $\sigma\in \Sigma$ (see  \cite{DGwonderful}).

\medskip
\noindent
Equivalently, every
$\sigma\in \Sigma$ has strongly convex image under the projection 
$$
N_{\bR} \to (N/N')_{\bR}.
$$

By \cite[Thm.\ 3.1]{DGwonderful},
property $(\mathrm{E})$ for $T'$, with respect to $\Sigma$, implies that
$X'$ is nonsingular, isomorphic to the toric variety
$X_{\Sigma'}$.
Then we have a closed immersion
\begin{equation}
\label{eqn.XSigmaprime}
X_{\Sigma'}\to X,
\end{equation}
given on affine charts, for $\sigma\in \Sigma'$, by ring homomorphisms
\[ k[\sigma^\vee \cap M]\to k[(\sigma^\vee \cap M)/(N'^\perp\cap M)]. \]
Generally, by \cite[Thm.\ 4.1]{DGwonderful}, after a suitable finite 
subdivision of $\Sigma$, one can insure 
property $(\mathrm{E})$ for a given subtorus, 
or any finite collection of subtori of $T$.

\subsection{Quotient tori and orbit closures}
\label{sect:qo}
A primitive sublattice $N'\subseteq N$ also gives rise to a quotient torus $T/T'$.
The case of interest to us is the quotient torus $T^\sigma$, associated with the sublattice $N_\sigma$ spanned by generators of a cone $\sigma\in \Sigma$; by the standing smoothness assumption on fans, $\sigma$ is generated by a part of a basis of $N$. 
Furthermore,
$\sigma$ determines an orbit $D_{\sigma}^\circ$,
the closed $T$-orbit in the corresponding affine chart $\Spec(k[\sigma^\vee\cap M])$ of $X$.
We denote its closure in $X$ by $D_{\sigma}$.
This is a smooth projective toric variety, whose fan is obtained as follows \cite[Thm. 3.2.6]{cox-book}:
\begin{itemize}
\item $\mathfrak X^*(T^\sigma) = {\sigma}^\perp \cap M$, which is dual to $N/N_{\sigma}$, 
\item for cones $\tau\supseteq \sigma$ let 
$$
\bar{\tau} := (\tau + \bR\sigma)/\bR \sigma \subseteq (N/N_{\sigma})_{\bR}
$$
be the induced cone in the quotient, these form 
a smooth projective fan $\Sigma^{\sigma}$.
\item We have a closed immersion $X_{\Sigma^{\sigma}}\hookrightarrow X$ with image $D_{\sigma}$.
On respective affine charts
$\Spec(k[\sigma^\perp\cap \tau^\vee\cap M])$ and
$\Spec(k[\tau^\vee\cap M])$,
for $\tau\supseteq \sigma$,
this is given by the surjective ring homomorphism
\[
k[\tau^\vee\cap M]\to k[\sigma^\perp\cap \tau^\vee\cap M],
\]
with kernel generated by characters in $\tau^\vee$, not in $\sigma^\perp$.
In particular, we have a canonical isomorphism
\[ T^{\sigma}=\Spec(k[\sigma^\perp\cap M])\cong D^\circ_{\sigma} \]
\item The projection $T\to T^{\sigma}$ determines a rational map $X\dashrightarrow X_{\Sigma^{\sigma}}$,
which is defined as a morphism on the
union $U^\sigma$ of affine charts of $X$ corresponding to cones $\tau\supseteq \sigma$.
This morphism
\begin{equation}
\label{eqn.smoothmorphism}
U^\sigma\to X_{\Sigma^\sigma}\cong D_\sigma
\end{equation}
is smooth, given by the injective ring homomorphisms
\[
k[\sigma^\perp\cap \tau^\vee\cap M]\to k[\tau^\vee\cap M].
\]
\end{itemize}

\subsection{Transversality of intersections}
\label{sect:ti}
By our assumption we have
\[ X\setminus T=\bigcup_{\rho\in \Sigma(1)}D_\rho, \]
a simple normal crossing divisor in $X$.
For $\sigma\in \Sigma(d)$, we have
$D_\sigma$ of codimension $d$ in $X$, with transverse intersection
\[
D_\sigma=\bigcap_{j=1}^{d} D_{\rho_j},
\]
where $\rho_j\in \Sigma(1)$,
$j=1$, $\dots$, $d$, are the rays spanning $\sigma$.

\begin{lemm}
\label{lem.closure}
Let $T'\subset T$ be a subtorus satisfying property $(\mathrm{E})$ with respect to $\Sigma$, and let $X'$ be the closure of $T'$ in $X=X_\Sigma$.
The morphism $T\to T/T'$ extends to a smooth $T$-equivariant morphism to $T/T'$ from a $T$-invariant neighborhood of $X'$ in $X$, with fiber $X'$ over $1\in T/T'$.
As well, $X'$
has transverse intersection with the boundary $X\setminus T$.
\end{lemm}

\begin{proof}
Let $N'\subseteq N$ be the corresponding primitive sublattice, with corresponding fan $\Sigma'$ in $N'_{\R}$.
By property $(\mathrm{E})$, the cones of $\Sigma'$ are already in $\Sigma$.
The variety $X'$ is contained in the union of affine charts of $X$ associated with
cones of $\Sigma$, that belong to $\Sigma'$, by the algebraic description of the closed immersion \eqref{eqn.XSigmaprime}.
This union is a $T$-invariant neighborhood of $X'$.
Let $\sigma$ be a maximal cone of $\Sigma'$.
Now we have a $T$-equivariant morphism
\[ \Spec(k[\sigma^\vee\cap M])\to \Spec(k[\sigma^\perp\cap M]), \]
extending $T\to T/T'$, and these patch to give the desired morphism, which is smooth.
We have
\[ \Spec(k[\sigma^\vee\cap M])\cong \A^{n'}\times \G_m^{n-n'}, \]
where $n'$ denotes the rank of $N'$.
Now
\[ X'\cap \Spec(k[\sigma^\vee\cap M])\cong \A^{n'}\times\{1\} \]
meets the complement of $\G_m^n$ transversely.
\end{proof}

\begin{prop}
\label{prop:trans}
Let $T'\subset T$ be a subtorus satisfying
property $(\mathrm{E})$ with respect to $\Sigma$,
and let $X'$ be the closure of $T'$ in $X$.
Let $\sigma\in \Sigma$ be such that 
the sublattice $N'\subset N$
corresponding to $T'$
satisfies 
$\sigma\subset N'_\R$.
Then:
\begin{itemize}
\item The intersection of $X'$ with $D_\sigma^\circ\cong T^\sigma$ is the subtorus $T'^{\sigma}\subset T^{\sigma}$ associated with
\[ N'/N_\sigma\subseteq N/N_\sigma. \]
\item The subtorus $T'^{\sigma}\subset T^{\sigma}$
satisfies property $(\mathrm{E})$ with respect to $\Sigma^\sigma$, and
the intersection of $X'$ with $D_\sigma\cong X_{\Sigma^\sigma}$ is
$X_{\Sigma'^\sigma}$, where $\Sigma'$ is the fan in $N'_\R$, induced from $\Sigma$, and $\sigma\in \Sigma'$ determines the fan $\Sigma'^\sigma$ in $(N'/N_\sigma)_\R$.
\item We have
$X'\cap U^\sigma$ equal to the
pre-image of $X_{\Sigma'^\sigma}$ under the smooth morphism \eqref{eqn.smoothmorphism}.
\end{itemize}
\end{prop}

\begin{proof}
In the affine chart $\Spec(k[\sigma^\vee\cap M])$ of $X$ associated with $\sigma$, we have closed subvarieties $D_\sigma^\circ$, defined by the characters in $\sigma^\vee\cap M$, not in $\sigma^\perp$, and the intersection with $X'$, given by equating characters whose difference lies in $N'^\perp\cap M$.
So the coordinate ring of the intersection is obtained by setting to zero all characters not in $\sigma^\perp$ and equating characters whose difference lies in $N'^\perp\cap M$;
this gives $k[(\sigma^\perp\cap M)/(N'^\perp\cap M)]$.
The first statement is established, since the subtorus of
$T^\sigma=\Spec(k[\sigma^\perp\cap M])$ associated with $N'/N_\sigma$ is
$\Spec(k[(\sigma^\perp\cap M)/(N'^\perp\cap M)])$.

For the second statement, we want to show, for $\tau\in\Sigma$ with $\tau\supset\sigma$, that $\bar\tau\cap (N'/N_\sigma)_\R$ is equal to $\bar\omega$ for some
$\omega\in\Sigma$ with $\omega\supset\sigma$.
We take $\omega:=\tau\cap N'_\R$;
since $T'$ satisfies property $(\mathrm{E})$ we have $\omega\in \Sigma$.
Since $\sigma\subset N'_\R$, we have $\omega\supset \sigma$.
Now
\[ (\tau+(N_\sigma)_\R)\cap N'_\R=\tau\cap N'_\R+(N_\sigma)_\R=\omega+(N_\sigma)_\R. \]
So $\bar\tau\cap (N'/N_\sigma)_\R=\bar\omega$, as desired.

We treat the remainder of the second statement, and the third, by an analysis on coordinate charts.
Let $\tau\in \Sigma'$, with $\tau\supset\sigma$.
In the affine chart $\Spec(k[\tau^\vee\cap M])$, we have $D_\sigma$ defined by characters in $\tau^\vee\cap M$, not in $\sigma^\perp$, and the intersection with $X'$, given by equating characters whose difference lies in $N'^\perp\cap M$.
This gives a description of the coordinate ring of the intersection as
\[ k[(\sigma^\perp\cap \tau^\vee\cap M)/(N'^\perp\cap M)]. \]
The same coordinate ring arises when we apply the description of the subtorus closure to $T'^\sigma\subset T^\sigma$.
By considering the same coordinate rings, and the injective ring homomorphisms corresponding to \eqref{eqn.smoothmorphism} for $\tau\in \Sigma$ and $\tau\in \Sigma'$, we obtain a commutative diagram of affine schemes which we see easily to be cartesian, and thereby obtain the third statement.
\end{proof}

\subsection{Equivariant combinatorics}
\label{sect:ec}
Let $G$ be a finite group.
A regular (right) $G$-action on a toric variety $X=X_{\Sigma}$ determines a representation 
\begin{equation}
\label{eqn:rep}
G\to \GL_n(\bZ)=\Aut(M).
\end{equation}
We are mainly interested in cases when this is {\em injective}.
The homomorphism \eqref{eqn:rep} determines a right action on the cocharacter lattice $N$.
The induced action on $N_\R$ leaves the
fan $\Sigma$ invariant, i.e., $\sigma\cdot g\in \Sigma$, for all $\sigma\in \Sigma$
and $g\in G$.

Since we are interested in equivariant birational types, we can start with a faithful representation \eqref{eqn:rep}.
By \cite{CTHS}, there exists a smooth projective fan $\Sigma$, that is invariant under the $G$-action.
Correspondingly, $G$ acts on $X_\Sigma$, with associated representation \eqref{eqn:rep}.

Suppose we are given a finite collection of subtori of $T$.
By combining \cite[Thm.\ 4.1]{DGwonderful} and \cite{CTHS}, we achieve, after a suitable subdivision of $\Sigma$, property $(\mathrm{E})$ for all of the subtori, still requiring $G$ to act regularly on the toric variety.
Notice that any further subdivision of $\Sigma$ preserves property $(\mathrm{E})$, for the given collection of subtori.

After a further $G$-equivariant subdivision of $\Sigma$, we may suppose, that the boundary $X\setminus T$ may be written as a union
$$
X\setminus T=\bigcup_{i\in \mathcal I} \, D_i, \qquad \mathcal I:=\{1,\dots,\ell\},
$$
where $D_i$ is a nonsingular $G$-invariant divisor, for all $i$.
This is achieved by iterated star subdivision along cones,
where the collection of generators is contained in an orbit of $\Sigma(1)$ and is maximal, with this property.

\section{Stabilizer stratification and subdivisions}

\subsection{Background}
\label{sect:background}
We recall basic terminology concerning {\em toric arrangements}, studied in, e.g., \cite{DP-toric}, \cite{DGP}, \cite{DG1}, \cite{DGwonderful}, \cite{DG3}, \cite{Callegaro}, \cite{berg}.

The main motivation for the introduction of the combinatorial formalism below was the computation of the cohomology ring of the {\em complement} of the arrangement; the relevant notions are: 
\begin{itemize}
\item Morgan differential algebra \cite{morgan},
\item Orlik-Solomon algebra \cite{orlik-solomon}.
\end{itemize}
We start with a torus $T$. The ingredients are:
\begin{itemize}
\item {\em Layers}, i.e., cosets of subtori of $T$.
\item {\em Toric arrangement}, i.e., finite collections of layers.
\item {\em Saturations} of toric arrangements, obtained by adding all connected components of intersections of layers.
\end{itemize}
The main result of \cite{DP-toric} and \cite{DG3} is that the {\em rational homotopy type} of the complement of a toric arrangement
only depends on discrete data associated with the toric arrangement.

Projective {\em wonderful models}  for toric arrangements are constructed in \cite{DGwonderful},
building on the wonderful models of subspace arrangements \cite{DP} and conical arrangements \cite{macphersonprocesi}.
These are defined as the closures of complements of toric arrangements in products of
\begin{itemize}
\item a nonsingular projective toric variety $X=X_\Sigma$, such that the subtori associated with the given layers satisfy property $(\mathrm{E})$ with respect to $\Sigma$, and
\item the blow-ups of $X$ along the closures of the layers.
\end{itemize}

\subsection{Arrangements of diagonalizable groups}
\label{sect:arr-diag}
The arrangements, relevant for our applications, are on the one hand less general than those in \S \ref{sect:background}, in that layers that arise are translates of subtori by \emph{torsion} elements of $T$.
On the other hand our inductive scheme requires arrangements in more general \emph{diagonalizable algebraic groups}, than just tori.

Let $\Delta$ be a diagonalizable algebraic group over $k$. The identity component $T\subset \Delta$ fits into an exact sequence
\begin{equation}
\label{eqn.DeltamodT}
1\to T\to \Delta \to \Delta/T\to 1,
\end{equation}
which is (noncanonically) split.
The subgroup $T$ is an algebraic torus, and the quotient $\Delta/T$ is a finite abelian group.

A (right) action of a
finite group $G$ on $\Delta$, by automorphisms, is determined uniquely by the data of a group homomorphism
\[ \mu\colon G\to \Aut(M), \quad M:=\mathfrak X^*(\Delta). \]
The induced action of $G$ on $T$ is determined by the induced homomorphism
\[ \nu\colon G\to \Aut(M/M_{\mathrm{tors}}). \]
There is also an induced action of $G$ on $\Delta/T$.
A $k$-point $\delta\in \Delta$ is given by a homomorphism $M\to k^\times$, that we also denote by $\delta$, and the corresponding class
$\bar\delta\in \Delta/T$ is given by the restriction of $\delta$ to $M_{\mathrm{tors}}$; so,
\[ \ker(\bar\delta)=\ker(\delta)\cap M_{\mathrm{tors}}. \]
There is an induced homomorphism
\[ \nu_{\bar\delta}\colon G_{\bar\delta}\to \Aut(M/\ker(\bar\delta)), \]
where $G_{\bar\delta}=\Stab(\bar\delta)\subseteq G$ denotes the stabilizer of $\bar\delta\in \Delta/T$.
We also remark that the $G$-action on $\Delta/T$ always fixes the identity; consequently, when $\Delta$ has nontrivial group of connected components, the action of $G$ cannot be transitive on components of $\Delta$.

That $G$ need not act transitively on components, represents a departure from the convention
in \cite[\S2]{KT-vector}.
For instance, then, different orbits of components can have different isomorphism types of generic stabilizer groups.
Nevertheless, we might agree to call the $G$-action on $\Delta$
\emph{generically free} when it satisfies the following
equivalent conditions.

\begin{lemm}
\label{lem.genericallyfree}
There exists a $G$-invariant dense open subscheme of $\Delta$, on which $G$ acts freely, if and only if $G$ acts generically freely on the identity component $T$ of $\Delta$.
\end{lemm}

\begin{proof}
A dense open subscheme of $\Delta$, on which $G$ acts freely, has nontrivial intersection with $T$ and exhibits the $G$-action on $T$ as generically free.
For the reverse implication, we let $r$ be a positive integer, such that the group of connected components of $\Delta$ is $r$-torsion.
Then the $r$th power endomorphism of $\Delta$ is $G$-equivariant and has image $T$.
The pre-image of a nonempty invariant open subscheme of $T$, on which $G$ acts freely, is an invariant dense open subscheme of $\Delta$, on which $G$ acts freely.
\end{proof}

\subsection{Lattice structure}
\label{sect:latt}
In the following discussion we do not assume that $G$ acts generically freely on $\Delta$.

We call a subgroup $\Gamma\subseteq G$
\emph{distinguished} if
$\Gamma$ is the largest subgroup,
acting trivially on some algebraic subgroup $\Theta\subset \Delta$.
We call an algebraic subgroup
$\Theta\subset \Delta$ \emph{distinguished} if $\Theta$ is maximal, on which $\Gamma$ acts trivially, for some subgroup $\Gamma\subseteq G$.
The operations, associating to
$\Theta$ the distinguished $\Gamma$,
and to $\Gamma$, the distinguished $\Theta$, restrict to inverse order-reversing bijections between distinguished subgroups of $G$ and distinguished algebraic subgroups of $T$.
The set of distinguished subgroups of $G$ has a structure of lattice,
with $\Gamma_1\wedge \Gamma_2=\Gamma_1\cap \Gamma_2$.
To describe $\Gamma_1\vee \Gamma_2$ we let $\Theta_i$ denote the algebraic subgroup of $\Delta$ associated with $\Gamma_i$ for $i=1$, $2$.
Then $\Gamma_1\vee \Gamma_2$ is the subgroup, containing both $\Gamma_1$ and $\Gamma_2$, with associated algebraic subgroup $\Theta_1\cap \Theta_2$.

Let $\Gamma$ be a distinguished subgroup, with associated algebraic subgroup $\Theta\subset \Delta$.
Then $\Gamma$ is the intersection of the generic stabilizer groups of the components of $\Theta$.
We might have $\Gamma$ as generic stabilizer of some component of $\Theta$, indeed this arises when we use the stabilizer of a point of $\Delta$ to define $\Theta$, whereby we see: \emph{every stabilizer group is distinguished}.
As the next example shows, a distinguished subgroup is not necessarily the stabilizer group of any point of $\Delta$.

\begin{exam}
\label{exam.notstabilizer}
Consider $G=C_2\times \fS_3$,
with $G\to \GL_2(\bZ)$ sending the generator of $C_2$, respectively generators of $\fS_3$, to
\[
\begin{pmatrix}
-1&0\\0&-1\end{pmatrix},\quad
\begin{pmatrix}
0&-1\\1&-1\end{pmatrix},\quad
\begin{pmatrix}
0&1\\1&0\end{pmatrix}.
\]
(This is essentially the unique faithful $2$-dimensional representation of $G$ over $\Z$ \cite{voskresenskii2dim}.)
The center $C_2$ is distinguished, associated with $(\mu_2)^2$. However, every point of $(\mu_2)^2$ has stabilizer of order $4$ or $12$.

%Consider the hyperoctahedral group of order $48$, 
%$G\subset \GL_3(\Z)$, with image generated by permutation and diagonal matrices.
%The normal subgroup of order $8$ is distinguished,
%associated with $(\mu_2)^3$.
%However, every point of $(\mu_2)^3$ has stabilizer of order $16$ or $48$.
\end{exam}

We modify the discussion above, by considering \emph{subtori} $T'\subset T$ and their associated distinguished subgroups $\Gamma'\subseteq G$.
The maximal algebraic subgroup $\Theta'\subset \Delta$, on which
$\Gamma'$ acts trivially, contains $T'$ and has the property, that every connected component has the same generic stabilizer.
The above bijection restricts to one, between distinguished subgroups of $\Gamma'\subseteq G$ associated with subtori of $T$ and distinguished algebraic subgroups of $\Delta$ whose components all have the same generic stabilizer; we write $T_{\Gamma'}$ for the identity component of the associated algebraic subgroup of $\Delta$.
There is a lattice structure, with
$\Gamma'_1\wedge \Gamma'_2=\Gamma'_1\cap \Gamma'_2$, but
now $\Gamma'_1\vee \Gamma'_2$ is the subgroup associated with the identity component of $T_{\Gamma'_1}\cap T_{\Gamma'_2}$.

We write
\[ \cL'=\cL'(T) \]
for the lattice of distinguished subgroups of $G$, associated with subtori of $T$.
The lattice $\cL'$ has maximal element $G$, and minimal element
$\ker(\nu)$.
To each $\Gamma'\in \cL'$, there is the subtorus $T_{\Gamma'}\subset T$, associated with a primitive sublattice $N_{\Gamma'}\subset N$; concretely, $N_{\Gamma'}$ is the sublattice of $N$, where $\Gamma'$ acts trivially.
We introduce
\[
\cT_{\cL'}:=\{T_{\Gamma'}\,|\,\Gamma'\in \cL'\},\qquad
\cG_{\cL'}:=\cT_{\cL'}\setminus \{T\}.
\]

\begin{lemm}
\label{lem.subtori}
Every distinguished algebraic subgroup $\Theta\subset \Delta$ has connected component of the identity in $\cT_{\cL'}$.
\end{lemm}

\begin{proof}
We let $\Gamma\subseteq G$ denote the distinguished subgroup associated with $\Theta$,
and $T'$, the identity component of $\Theta$.
The generic stabilizer $\Gamma'$ of $T'$ contains $\Gamma$, and we have $\Gamma'\in \cL'$.
We have $T'$ contained in $T_{\Gamma'}$, the identity component of the associated algebraic subgroup $\Theta'$.
Since $T_{\Gamma'}\subset \Theta'\subset \Theta$, we have $\dim T_{\Gamma'}\le \dim T'$.
It follows that $T'=T_{\Gamma'}$.
\end{proof}

Notice, by Lemma \ref{lem.subtori}, the
locus in $\Delta$ with nontrivial stabilizers is contained in a finite union of translates of tori $T_{\Gamma'}$ for $\Gamma'\in \cL'$, where the translates are by torsion elements of $\Delta$.
(For any stabilizer group, by Lemma \ref{lem.subtori} the associated distinguished algebraic subgroup will be such a finite union.)
As well, $T_{\Gamma'}$ is invariant under the action of $g\in G$ if and only if $g\in N_G(\Gamma')$.
(The generic stabilizer gets conjugated by $g$.)

\subsection{Equivariant compactifications of diagonalizable groups}
\label{sect:ecdg}
Let $\Delta$ be a diagonalizable algebraic group over $k$, with an action of a finite group $G$.
Since the character groups $\mathfrak X^*(T)$ and $\mathfrak X^*(\Delta)$ have a common dual $N$,
a smooth projective fan $\Sigma$ in $N_\R$ determines,
besides the equivariant compactification
\[ T\subset X=X_\Sigma, \]
also an equivariant compactification
%\[ \Delta\subset X^+=X^+_\Sigma. \]
\[ \Delta\subset \mathbb X=\mathbb X_\Sigma. \]
The smooth projective scheme $\mathbb X$ has components indexed by
$\Delta/T$:
\[
\mathbb X=\bigsqcup_{\bar{\delta} \in \Delta/T} \,  X\bar\delta.
\]

Let $r$ be a positive integer, such that $\Delta/T$ is $r$-torsion.
A primitive sublattice $N'\subseteq N$ now gives rise to:
\begin{itemize}
\item The subtorus $T'\subset T$, its $r$-torsion translate
\[ T'_{[r]}:=T'\Delta[r]\subset \Delta, \]
and the induced surjective homomorphism
\[
\vartheta_{[r]}\colon T'_{[r]}/T'\to \Delta/T.
\]
\item The corresponding equivariant compactification
\[ X'_{[r]}:=X'\Delta[r]\subset \mathbb X_\Sigma, \]
with components indexed by $T'_{[r]}/T'$.
\item The quotient diagonalizable algebraic group
\[ \Delta/T', \]
which as in
Lemma \ref{lem.closure} is the target of a smooth $\Delta$-equivariant morphism
from a
$\Delta$-invariant neighborhood of $X'_{[r]}$ in $\mathbb X_\Sigma$,
with $X'_{[r]}$ as fiber over $(\Delta/T')[r]$.
\item In case $N'=N_\sigma$,
notation $\Delta^\sigma$ for the quotient diagonalizable algebraic group, with
$\Delta^\sigma$ identified with a corresponding $\Delta$-orbit in $\mathbb X_\Sigma$ as in \S \ref{sect:qo}, and closure of $\Delta^\sigma$ in $\mathbb X_\Sigma$ identified with $\mathbb X_{\Sigma^\sigma}$.
\end{itemize}

\subsection{Constructing the model}
\label{sect:cm}
Let $\cL'=\cL'(T)$ be the lattice of distinguished subgroups of $G$, associated with subtori of $T$, as above.
As we have seen in Section \ref{sec.generalities},
there exists a smooth projective fan $\Sigma$ that
meets the following criteria:
\begin{itemize}
\item Property $(\mathrm{E})$ holds for all subtori in $\cT_{\cL'}$, with respect to $\Sigma$.
\item $\Sigma$ is $G$-invariant.
\item No pair of rays of $\Sigma$, in a single $G$-orbit, spans a cone of $\Sigma$.
\end{itemize}

For $\sigma\in \Sigma$,
the stabilizer $\Stab(\sigma)$ acts on
the toric variety
$D_\sigma\cong X_{\Sigma^\sigma}$, an equivariant compactification of the quotient torus $T^\sigma$.
%We claim that 

\begin{lemm}
\label{lemm:preim}
We have $\Stab(\sigma)\in \cL'$.
The lattice $\cL'(T^\sigma)$, associated with the $\Stab(\sigma)$-action on $T^\sigma$, is equal to the sublattice of $\cL'$, of elements bounded above by $\Stab(\sigma)$.
For $\Gamma'\in \cL'(T^\sigma)$, the
pre-image of $(T^\sigma)_{\Gamma'}$ under the projection morphism
$$
\mathrm{pr}^\sigma\colon T\to T^\sigma
$$
is equal to $T_{\Gamma'}$.
\end{lemm}

\begin{proof}
%To prove the claim, 
By the third criterion on $\Sigma$, the action of $\Stab(\sigma)$ on $N_\sigma$ is trivial.
Consequently, if $\Gamma'\in \cL'(T^\sigma)$, so $\Gamma'$ is the generic stabilizer of $(T^\sigma)_{\Gamma'}$,
then $\Gamma'$ acts trivially on
$(\mathrm{pr}^\sigma)^{-1}((T^\sigma)_{\Gamma'})$.
Using Proposition \ref{prop:trans}, we see that $\Gamma'$ is the generic stabilizer of $(\mathrm{pr}^\sigma)^{-1}((T^\sigma)_{\Gamma'})$, and
$T_{\Gamma'}=(\mathrm{pr}^\sigma)^{-1}((T^\sigma)_{\Gamma'})$.
In the other direction, if $\Gamma'\in \cL'$, $\Gamma'\subseteq \Stab(\sigma)$, then
$T_{\Gamma'}$ is the pre-image under $\mathrm{pr}^\sigma$ of a subtorus of $T^\sigma$, whose generic stabilizer is $\Gamma'$.
\end{proof}

Each $\Delta$-orbit of $\mathbb X_\Sigma$ gives an instance of \S \ref{sect:arr-diag}, with $\Stab(\sigma)$ acting on $\Delta^\sigma$,
and the locus with nontrivial stabilizer contained in some translates of the subtori in $\cT_{\cL'(T^\sigma)}$.
As noted above, the following are valid:
\begin{itemize}
\item For a suitable positive integer $r$, the translates of the subtori in $\cT_{\cL'(T^\sigma)}$ are by $r$-torsion elements of $\Delta^\sigma$.
\item The subtori in $\cT_{\cL'(T^\sigma)}$ have pre-images in $T$, belonging to $\cT_{\cL'}$.
\item Property $(\mathrm{E})$ holds for the subtori in $\cT_{\cL'(T^\sigma)}$, with respect to $\Sigma^\sigma$.
\end{itemize}
Since $\Sigma$ has finitely many cones, a single positive integer $r$ may be chosen, so that translation
is by $r$-torsion elements of $\Delta^\sigma$, in the first item above, for all $\sigma\in \Sigma$.
We suppose, as well, that $\Delta/T$ is $r$-torsion.

Consequently, the intersection of any pair of subtori in $\cT_{\cL'}$ is a diagonalizable algebraic group, whose group of connected components is $r$-torsion.
Indeed, if $\Gamma'$, $\Gamma''\in \cL'$, with respective associated subtori $T':=T_{\Gamma'}$ and $T'':=T_{\Gamma''}$, then
$T''':=T_{\Gamma'\vee\Gamma''}$ is the identity component of $T'\cap T''$.
Any non-identity component of $T'\cap T''$ has generic stabilizer contained in $\Gamma'\vee\Gamma''$ and
associated distinguished algebraic subgroup $\Theta\subset T$, that satisfies
\[ T'''\subset \Theta\subset T'''\Delta[r]. \]
So $(T'\cap T'')/T'''$ is $r$-torsion.

Inside $\mathbb X=\mathbb X_\Sigma$ there is the union of $r$-torsion translates of the closures of the subtori in $\cG_{\cL'}$.
The complement, the ``open part''
\[
\mathbb X^{\circ}\subset \mathbb X,
\]
has the stabilizer distinguishing property, for the $G$-action with respect to the toric boundary.

The projective model
\[ \mathbb X_{\Sigma,\cL',[r]} \]
is obtained as in \cite{DGwonderful}, by applying the De Concini-Procesi iterated blowup procedure, as developed in \cite{li-wonderful}, to the $r$-torsion translates of the closures in $X_{\Sigma}$ of the subtori in $\cG_{\cL'}$.
So, $\mathbb X_{\Sigma,\cL',[r]}$ is the closure of $\mathbb X^{\circ}$, in the product of $\mathbb X$ with all blow-ups
\[ B\ell_{X'_{[r]}} \mathbb X \]
for $T'\in \cG_{\cL'}$, where
$X'_{[r]}\subset \mathbb X$ denotes the corresponding $r$-torsion translate compactification (\S \ref{sect:ecdg}).
There is a projection morphism
\begin{equation}
\label{eqn.projection}
\pi\colon \mathbb X_{\Sigma,\cL',[r]}\to \mathbb X,
\end{equation}
which is an isomorphism over $\mathbb X^{\circ}$.
As is the case for $\mathbb X$, the projective model has connected components indexed by $\Delta/T$:
\[ \mathbb X_{\Sigma,\cL',[r]}=\bigsqcup_{\bar\delta\in \Delta/T} X_{\Sigma,\cL',[r]}\bar\delta. \]
In particular, when $\Delta$ is a torus already, we have
$\mathbb X_{\Sigma,\cL',[r]}=X_{\Sigma,\cL',[r]}$.
%i.e., we are free to omit the superscript $+$ from the notation.

The complement of $\mathbb X^{\circ}$ in $\mathbb X_{\Sigma,\cL',[r]}$ is a normal crossing divisor
\[
\mathbb D=\bigcup_{\ker(\nu)\ne\Gamma'\in \cL'}\mathbb D_{\Gamma'}.
\]
For $\Lambda\subset \cL'\setminus\{\ker(\nu)\}$, there is the stratum
\[
\mathbb D_{\Lambda}:=\bigcap_{\Gamma'\in \Lambda}\mathbb D_{\Gamma'},
\]
with $\mathbb D_\emptyset:=\mathbb X_{\Sigma,\cL',[r]}$ by convention.
We have $\mathbb D_\Lambda\ne\emptyset$ if and only if
$\Lambda$ is a chain in $\cL'$, i.e., letting $t$ denote the cardinality of $\Lambda$ we have
\begin{equation}
\label{eqn.chain}
\Lambda=\{\Gamma^1,\dots,\Gamma^t\}\subset \cL'\setminus\{\ker(\nu)\},\qquad \Gamma^1\supset\dots\supset \Gamma^t.
\end{equation}
Suppose $\Lambda$ is a nonempty chain.
Then with
\begin{equation}
\label{eqn.chainfirstT}
T':=T_{\Gamma^1},
\end{equation}
the connected components of $\mathbb D_\Lambda$ are indexed by $T'_{[r]}/T'$:
\begin{equation}
\label{eqn.DLambdacomponents}
\mathbb D_\Lambda=\bigsqcup_{\bar\tau\in T'_{[r]}/T'} D_\Lambda \bar\tau.
\end{equation}

\begin{exam}
\label{exa.Z2onBlpP2}
Let $G=\Z/2\Z$ act on $T=\G_m^2$, swapping the two factors.
With $\Sigma$ the complete fan in $N_\R$, for $N=\Z^2=\Z\langle e_1,e_2\rangle$, with rays generated by $e_1$, $e_1+e_2$, $e_2$, $-e_1-e_2$, so $X_\Sigma$ is the blow-up of $\bP^2$ at a point, the origin in $\A^2\subset\bP^2$.
We have $\cL'=\{\mathrm{triv},G\}$, with respective subtori $\G_m^2$ and $\Delta_{\G_m}$ (the diagonal).
The stabilizer locus in $T$ is precisely $\Delta_{\G_m}$.
However, we need to take $r$ divisible by $2$, since $\sigma=\R_{\ge 0}\cdot (e_1+e_2)$
(and also $\sigma=\R_{\ge 0}\cdot (-e_1-e_2)$)
leads to $\Stab(\sigma)=G$ acting nontrivially on a one-dimensional torus;
this fixes the $2$-torsion subgroup.
We only blow up divisors, so
\[ X_{\Sigma,\cL',[2]}\cong X_\Sigma, \]
and $X^\circ$ is the complement of the proper transforms of two lines in $\bP^2$, intersecting at the origin in $\A^2\subset \bP^2$, with slopes $\pm 1$.
\end{exam}
 
 \subsection{Properties of the model}
 \label{sect:propmod}

\begin{prop}
\label{prop.standardmodel}
Suppose that $G$ acts generically freely on $\Delta$.
Then,
for $\Sigma$ and $r$ as above, the projective variety $$
\mathbb X_{\Sigma,\cL',[r]}
$$ 
is in standard form
with respect to the union of the strict transform of the toric boundary and the exceptional divisors of the De Concini-Procesi iterated blowup procedure.
\end{prop}

To be in \emph{standard form} with respect to a simple normal crossing divisor, means that the divisor has smooth $G$-orbits of components and that $G$ acts freely on the complement  \cite{reichsteinyoussinessential}.

\begin{proof}
The exceptional divisors of the
De Concini-Procesi blowup form a simple normal crossing divisor.
The centers of blowup $X'_{[r]}$ have transverse intersection with the toric boundary by Lemma \ref{lem.closure}.
By Proposition \ref{prop:trans}, the exceptional divisors, together with the proper transform of the toric boundary, form a simple normal crossing divisor.
We conclude by Lemma \ref{lem.subtori}.
\end{proof}

By analogy with the treatment in \cite{FK3}, we describe a point $p\in \mathbb X_{\Sigma,\cL',[r]}$
as a pair
\[ p=(x,V_1\subset W_1\subset\dots\subset V_t\subset W_t) \]
consisting of a point $x\in \mathbb X$, say in the $\Delta$-orbit identified with $\Delta^\sigma$ (\S \ref{sect:ecdg}), and
a flag of subspaces of $N_k:=N\otimes k$ such that
\begin{itemize}
\item $V_i=N_{\Gamma^i}\otimes k$ for some $\Gamma^i\in \cL'$, for all $i$.
\item We have $\dim(W_i)=\dim(V_i)+1$ for all $i$.
\item The maximal $\Gamma'\in \cL'(T^\sigma)$ with $x\in (T^\sigma)_{\Gamma'}\Delta^\sigma[r]$ is $\Gamma^1$, when $t\ge 1$, otherwise is $\ker(\nu)$.
\item The maximal $\Gamma'\in \cL'$ with $W_i\subset N_{\Gamma'}\otimes k$, is $\Gamma^{i+1}$, for $i<t$.
\item If $t\ge 1$, the maximal $\Gamma'\in \cL'$ with $W_t\subset N_{\Gamma'}\otimes k$, is $\ker(\nu)$.
\end{itemize}
The point $p$ lies in the stratum $\mathbb D_\Lambda$ indexed by the chain
\[
\Gamma^1\supset\dots\supset \Gamma^t,
\]
and not in any deeper stratum.
The stabilizer of $p$ is
\[ \{g\in G\,|\,x\cdot g=x,\text{ and }V_i\cdot g=V_i \text{ and }W_i\cdot g=W_i\text{ for all }i\}. \]

We take $T'$ as in
\eqref{eqn.chainfirstT}, with
corresponding sublattice $N':=N_{\Gamma^1}$.
So, the tangent space to $T/T'$ at the identity is naturally identified with $(N/N')_k$.
The normalizer $N_G(\Gamma^1)$ acts on the quotient $\Delta/T'$, and the smooth morphism from a $\Delta$-invariant neighborhood of $X'_{[r]}$ to $\Delta/T'$ with fiber $X'_{[r]}$ over $(\Delta/T')[r]$, mentioned in \S \ref{sect:ecdg}, is $N_G(\Gamma^1)$-equivariant.
Denoting such a $\Delta$-invariant neighborhood by $Q$, we have the composite
\begin{equation}
\label{eqn.compositefromV}
Q\to \Delta/T'\stackrel{r\cdot}\to \Delta/T',
\end{equation}
with image $T/T'$, and pre-image $X'_{[r]}$ of the identity element.

\begin{lemm}
\label{lem.divisors}
Let the notation be as above.
\begin{itemize}
\item[(i)]
The projection morphism $\pi$, of introduced in \eqref{eqn.projection}, factors through $B\ell_{X'_{[r]}} \mathbb X$.
\item[(ii)]
There exist an $N_G(\Gamma^1)$-invariant neighborhood $W\subset T/T'$ of the identity and an \'etale $N_G(\Gamma^1)$-equivariant morphism from $W$ to the vector space $(N/N')_k$, with fiber over $0$ consisting just of the identity element of $T/T'$, where the corresponding map of tangent spaces gives the natural identification with $(N/N')_k$.
\item[(iii)]
If we let $U$ denote the pre-image of $W$ under the composite morphism \eqref{eqn.compositefromV}, then by following the composite \eqref{eqn.compositefromV} with the morphism of $\mathrm{(ii)}$, we get a smooth $N_G(\Gamma^1)$-equivariant morphism
\[ U\to (N/N')_k, \]
with pre-image of $0$ equal to $X'_{[r]}$.
\item[(iv)] For the induced equivariant morphism $B\ell_{X'_{[r]}}U\to \bP((N/N')_k)$, the class of the exceptional divisor in $\Pic^{N_G(\Gamma^1)}(B\ell_{X'_{[r]}}U)$ is the pullback of $\cO_{\bP((N/N')_k)}(-1)$.
The pullback of $\cO_{\bP((N/N')_k)}(-1)$ to $\pi^{-1}(U)$ is the class of the divisor
\[
\bigcup_{\substack{\Gamma'\in \cL'\\ \Gamma^1\subseteq \Gamma'}} \mathbb D_{\Gamma'},
\]
with all components of multiplicity $1$.
\end{itemize}
\end{lemm}

\begin{proof}
With $\mathbb X_{\Sigma,\cL',[r]}$ as the closure of $ \mathbb X^{\circ}$ in a product of blow-ups, projection to $B\ell_{X'_{[r]}} \mathbb X$ yields the factorization in (i).

We get (ii) from an equivariant version of a construction of Moci \cite{moci}.
The coordinate ring $k[(N'^\perp\cap M)/M_{\mathrm{tors}}]$ of $T/T'$ has a maximal ideal $\mathfrak{m}$, corresponding to the identity element of $T/T'$.
A splitting of the surjective $N_G(\Gamma^1)$-equivariant homomorphism
\[ \mathfrak{m}\to \mathfrak{m}/\mathfrak{m}^2 \]
determines $N_G(\Gamma^1)$-equivariant $T/T'\to (N/N')_k$, sending the identity to $0$.
The corresponding map of tangent spaces gives the natural identification of the tangent space to $T/T'$ at the identity with $(N/N')_k$.
Then for suitable $W$ we have
(ii); an immediate consequence is (iii).

Since blowing up commutes with smooth base change, we get the first assertion in (iv) from the standard description of $B\ell_{\{0\}}(N/N')_k$ as the closure of the complement of $0$ in
$(N/N')_k\times \bP((N/N')_k)$.
For the remaining assertion, we use
the treatment of iterated blow-ups in \cite{li-wonderful}, which lets us express $\pi$ as an iterated blow-up of $\mathbb X$ in the following manner.
The first step is the blow-up
\[ \pi_1\colon B\ell_{X'_{[r]}}\mathbb X\to \mathbb X. \]
In subsequent steps we
blow up the (proper transforms of the) pre-images under $\pi_1$ of $r$-torsion translate compactifications $X''_{[r]}\subsetneq X'_{[r]}$ in any order of weakly increasing dimensions.
Finally we blow up the proper transforms of the remaining $r$-torsion translate compactifications, in any order of weakly increasing dimensions.
Examination of the behavior of the exceptional divisor of $\pi_1$ under these blow-ups gives what we need.
\end{proof}

\section{Equivariant Burnside group}
\label{sect:ebg}
In this section we recall the equivariant Burnside group, introduced in
\cite{BnG}, and state the formula, that we will use for our computation in the equivariant Burnside group.

The \emph{equivariant Burnside group}
\[ \Burn_n(G)=\Burn_{n,k}(G) \]
is an abelian group, generated by
symbols
\[ (H,Y\actsfromleft K,\beta), \]
where
\begin{itemize}
\item $H\subseteq G$ is an abelian subgroup,
\item $K$ is a field, finitely generated over $k$, with faithful action over $k$ of a subgroup $Y\subseteq Z:=Z_G(H)/H$, where $Z_G(H)$ denotes the centralizer of $H$ in $G$,
\item $\beta$ is a sequence of length $r:=n-\mathrm{trdeg}_{K/k}$ of
nontrivial characters of $H$, that generates $H^\vee$.
\end{itemize}
The symbols are subject to relations, labeled
$$
{\bf{(O)}}, {\bf{(C)}}, 
{\bf{(B1)}}, \text{ and } {\bf{(B2)}}
$$
(which stand for ordering, conjugation, and blowup relations), 
e.g., the equivalence of symbols that differ by a re-ordering of the sequence of characters.
Symbols are also permitted in which $K$ is a Galois algebra for some $Y\subseteq Z$ over a field that is finitely generated over $k$; we identify
$(H,Y\actsfromleft K,\beta)$ with $(H,Z\actsfromleft \Ind_Y^Z(K),\beta)$.
See \cite{KT-vector} for a complete description of relations.

In a symbol, $\beta$ is determined uniquely, up to order, by the similarity type of a
faithful
$(n-\mathrm{trdeg}_{K/k})$-dimensional representation of $H$ over $k$, or any field containing $k$.
Furthermore, we declare a symbol to be trivial
in case the trivial character occurs, i,e., if the given representation has nontrivial space of invariants.

We record a frequently used consequence of defining relations \cite[Prop. 4.7]{BnG}:
If $\beta=(b_1,\ldots, b_r)$ is such that for some 
$$
I\subseteq [1,\ldots, r], \quad |I|\ge 2,
$$
one has
$$
\sum_{i\in I} b_i = 0,$$
then 
$$
(H,Y\actsfromleft K,\beta) = 0\in \Burn_n(G).
$$
As an immediate application, we obtain: 

\begin{prop}
\label{prop:s-vanishing}
Consider the symbol
$$
(H, Y\actsfromleft K, \beta)\in \Burn_n(G), \quad \beta=(b_1,\ldots, b_r), \quad 1\le r\le n. 
$$
Let 
$$
\ell:=\min_{1\le j\le r}(\mathrm{ord}(b_j))
$$
be the smallest order of a character appearing in $\beta$. 
We have
$$
(H, Y \actsfromleft K(t_1,\ldots, t_{\ell-1}), \beta)=0 \in\Burn_{n+\ell-1}(G).
$$
Here $Y$ acts trivially on the variables $t_1,\ldots t_{\ell-1}$. 
\end{prop}

\begin{proof}
We may assume that the minimum is at $b_1$. Consider the symbol
$$
(H, Y \actsfromleft K, (\underbrace{b_1, \ldots, b_{1}}_{\ell \text{ times }},b_2,\ldots, b_r)) \in \Burn_{n+\ell-1}(G).
$$
This symbol vanishes. 
Applying relation {\bf{(B2)}} iteratively, we obtain the claim. 
\end{proof}

A nonsingular projective variety $X$ with generically free $G$-action determines, as $G$-equivariant birational invariant, a class
\[ [X\actsfromright G]\in \Burn_n(G). \]
This is particularly easy to describe, when $X$ is in standard form with respect to a $G$-invariant simple normal crossing divisor, or more generally satisfies Assumption 2 of \cite{BnG}.
Then
\begin{align*}
[&X\actsfromright G]=\\
&\,\, \sum_{\substack{x_0\in X/G\\ x_0=[x],\,x\in X}}
\big(\text{generic stabilizer of $\overline{\{x\}}$},
\Gal(k(x)/k(x_0)),
(\cI_{\overline{\{x\}}}/\cI_{\overline{\{x\}}}^2)_x\big).
\end{align*}
Points of the quotient $X/G$ are in bijective correspondence with $G$-orbits of points of $X$; the sum is over $x_0\in X/G$, and $x$, in the sum, denotes an orbit representative.
Importantly, not only $k$-points, but all points are taken in the sum.
The residue field $k(x)$ is a Galois extension of $k(x_0)$.
The ideal sheaf
$\cI_{\overline{\{x\}}}$ of
the closure
$\overline{\{x\}}$
defines the coherent sheaf
$\cI_{\overline{\{x\}}}/\cI_{\overline{\{x\}}}^2$ on
$\overline{\{x\}}$, whose stalk at $x$ gives a faithful representation of the generic stabilizer of $\overline{\{x\}}$.
In order to get a representation with trivial space of invariants, $x$ has to be a maximal point with this generic stabilizer, so only finitely many points $x_0\in X/G$ yield nontrivial symbols.

\begin{rema}
\label{rema:vanishing}
An immediate corollary of Proposition~\ref{prop:s-vanishing} is that the symbols invariant is stably trivial: 
For any $G$-variety $X$ of dimension $n$, there exists a $d\in \bN$ such that class 
$$
[X\times \bP^{d}\actsfromright G] -(\mathrm{triv}, G\actsfromleft k(X)(t_1,\ldots, t_{d}), ()) =0 \in \Burn_{n+d}(G);  
$$
here $G$-acts trivially $\bP^d$, respectively, on the variables $t_1,\ldots, t_d$. 

This is in stark contrast with invariants originating in unramified cohomology. 
\end{rema}

When $X$ is replaced by a $G$-invariant open subvariety $U$, the same formula leads to a class in $\Burn_n(G)$, that is denoted by
\[ [U\actsfromright G]^{\mathrm{naive}}. \]
(A class $[U\actsfromright G]$ is also defined in \cite{BnG}, but plays no role in this paper.)

Let
\[ D=D_1\cup\dots\cup D_\ell \]
be a simple normal crossing divisor on $X$, where each $D_i$ is $G$-invariant and nonsingular.
Then, with $\cI:=\{1,\dots,\ell\}$, we have
\[
[X\actsfromright G]=[U\actsfromright G]^{\mathrm{naive}}+
\sum_{\emptyset\ne I\subseteq \cI}
\sum_{j\in \cJ_I}
\mathrm{ind}_{G_{I,j}}^G\big(\psi_I([\underline{D}^\circ_{I,j}\actsfromright G_{I,j}]^{\mathrm{naive}}_{(\mathcal{N}_{D_i/X})_{i\in I}})\big).
\]
Here,
\begin{itemize}
\item $\cJ_I$ indexes $G$-orbits of components of $D_I:=\bigcap_{i\in I}D_i$, with notation $D_{I,j}$ for the $G$-orbit of components corresponding
to $j\in \cJ_I$, and $\underline{D}_{I,j}$ for a chosen component of $D_{I,j}$.
\item $G_{I,j}$ is the maximal subgroup of $G$, for which $\underline{D}_{I,j}$ is invariant.
\item We denote by $D^\circ_I\subset D_I$, by $D^\circ_{I,j}\subset D_{I,j}$, and by
$\underline{D}^\circ_{I,j}\subset \underline{D}_{I,j}$,
the complement of all $D_j$, $j\notin I$.
\item The class 
$$
[\underline{D}^\circ_{I,j}\actsfromright G_{I,j}]^{\mathrm{naive}}_{(\mathcal{N}_{D_i/X})_{i\in I}} \in \Burn_{n,I}(G_{I,j}), 
$$ in the \emph{indexed equivariant Burnside group} 
\cite[\S 4]{KT-struct} \cite[\S 4]{KT-vector} is defined like a naive class in $\Burn_n(G_{I,j})$, but with additional data of characters associated with each of the indicated line bundles \cite[\S 5]{KT-vector} (in this case, the restrictions of the normal bundles $\mathcal{N}_{D_i/X}$).
\item The homomorphism
$$
\psi_I\colon \Burn_{n,I}(G_{I,j})\to \Burn_n(G_{I,j})
$$ 
appends the characters associated with the line bundles to the characters $\beta$ \cite[Rem.\ 4.1]{KT-struct}.
\item 
The induction homomorphism
$$
\mathrm{ind}_{G_{I,j}}^G : \Burn_n(G_{I,j}) \to \Burn_n(G)
$$
is defined in 
\cite[Defn.\ 3.1]{KT-vector}.
\end{itemize}
This formula holds by \cite[Prop.\ 4.8]{KT-vector}; see, also \cite[Exa.\ 5.12]{KT-vector}.

The additional characters, attached to symbols in the indexed equivariant Burnside group, may be manipulated, e.g., by applying any element of $\Aut(\Z^I)$, where $\Z^I$ denotes $\bigoplus_{i\in I}\Z$.
For instance, when $J\subseteq I$ and we apply the element
\[
\tau_{I,J}\in \Aut(\Z^I), \qquad \tau_{I,J}(e_j):=\begin{cases}
{\displaystyle\sum_{\substack{i\in I,\,i\le j\\ i>j'\,\forall\, j'\in J,\, j'<j}}e_i},&\text{if $j\in J$},
\\
\ \ \ \ \ \ e_j,&\text{if $j\notin J$},
\end{cases}
\]
of \cite[Exa.\ 4.1]{KT-vector},
to a symbol $(H\subseteq H',Y\actsfromleft K,\beta,\gamma)\in \Burn_{n,I}(G)$ with the additional characters
\[ \gamma=(c_i)_{i\in I}, \]
we get the symbol
$\tau_{I,J}(H\subseteq H',Y\actsfromleft K,\beta,\gamma):=(H\subseteq H',Y\actsfromleft K,\beta,\tilde\gamma)$,
\[ \tilde\gamma=(\tilde c_i)_{i\in I}, \]
where $\tilde c_j$ is the sum of $c_i$ over $i\le j$ with $i>j'$ for all $j'\in J$, $j'<j$ for $j\in J$, and $\tilde c_j=c_j$ otherwise.

As a generalization of the map $\psi_I$ there is the map
\[ \psi_{I,J}\colon \Burn_{n,I}(G)\to \Burn_{n,J}(G), \]
which appends just the characters indexed by elements of $I\setminus J$ to the characters $\beta$; see \cite[Defn.\ 4.2]{KT-vector}.

\section{Computing the class in the Burnside group}
\label{sect:class}
We wish to compute a class in an equivariant Burnside group, associated with a given $G$-action on $\Delta$. 
%Generally, $\Delta$ may have several components, 
%and we obtain a collection of classes, indexed by components $\Delta/T$.
We have the exact sequence of algebraic groups \eqref{eqn.DeltamodT} and the notation introduced in \S \ref{sect:arr-diag}.
%\[
%\begin{array}{ccccccccc}
%  &     &    &  & \delta    &   \mapsto  &    \bar{\delta} &     &   \\
%  &     &    &  & \vin    &          &    \vin &     &   \\
%1 & \to &  T & \to & \Delta & \to & \Delta/T        & \to  & 1 \\ 
%\end{array}
%\]
We fix $\bar\delta\in \Delta/T$, the class of some $\delta\in \Delta$.
Besides the stabilizer $G_{\bar\delta}$ of $\bar\delta$, there is
\[ G_\delta:=\ker(\nu_{\bar\delta}), \]
which is equal to the stabilizer of $\delta\in \Delta$, provided that $\delta$ is a suitably general lift of $\bar\delta\in \Delta/T$.
The group $G_\delta$ is the generic stabilizer of the induced action
of $G_{\bar\delta}$ on the component $T\bar\delta$ of $\Delta$.

We recall, $\bar\delta\in \Delta/T$ indexes a component $X\bar\delta$ of $\mathbb X$.
This has a class
\begin{equation}
\label{eqn.classdeltacomponent}
[X\bar\delta \actsfromright G_{\bar\delta}/G_\delta]\in \Burn_n(G_{\bar\delta}/G_\delta).
\end{equation}
Our goal is to compute the class of $\mathbb X\actsfromright G$, which we understand to mean the collection of classes \eqref{eqn.classdeltacomponent} for all $\bar\delta\in \Delta/T$.

We recall the lattice $\cL'=\cL'(T)$ from \S \ref{sect:latt} and make a choice of fan $\Sigma$ and positive integer $r$ as in \S \ref{sect:cm}; we work on the model $\mathbb X_{\Sigma,\cL',[r]}$.
Exactly as in Proposition \ref{prop.standardmodel}, for the action of $G_{\bar\delta}/G_\delta$ we have $X_{\Sigma,\cL',[r]}\bar\delta$ in standard form, with respect to the union of the strict transform of the toric boundary and the exceptional divisors of the De Concini-Procesi iterated blowup procedure.

We introduce $G_{\bar\delta}$-invariant divisors.
Let
\[ \Gamma'_1,\dots,\Gamma'_{\ell_{\bar\delta}}\in \cL' \]
be a choice of conjugacy class representatives of elements of $\cL'\setminus \{\ker(\nu)\}$.
We set
\[ \mathbb D_i:=\bigcup_{\Gamma'\text{ conjugate to }\Gamma'_i} \mathbb D_{\Gamma'}. \]
Then $\mathbb D=\mathbb D_1\cup\dots\cup \mathbb D_{\ell_{\bar\delta}}$ is a
simple normal crossing divisor on $\mathbb X_{\Sigma,\cL',[r]}$, with each $\mathbb D_i$ invariant under $G_{\bar\delta}$.
Now suppose 
$$
I\subseteq \cI_{\bar\delta}:=\{1,\dots,\ell_{\bar\delta}\},
$$
with $\mathbb D_I$ nonempty.
We take $\cJ_I$ to be the set of conjugacy classes of chains in $\cL'$, with one element from each conjugacy class, indexed by an element of $I$.
Then $j\in \cJ_I$ indexes an orbit $\mathbb D_{I,j}$ of $\mathbb D_\Lambda$ for a representative chain $\Lambda$ of the conjugacy class of $j$.
The maximal subgroup, under which $\mathbb D_\Lambda$ is invariant, is
\[ N_{G_{\bar\delta}}(\Lambda):=N_{G_{\bar\delta}}(\Gamma^1)\cap\dots\cap N_{G_{\bar\delta}}(\Gamma^t), \]
the stabilizer of $\Lambda$ under the conjugation action of ${G_{\bar\delta}}$.

With the subspaces $V_i=N_{\Gamma^i}\otimes k$ of
\[ V:=N\otimes k \]
for $1\le i\le t$, appearing in the description of a point of $\mathbb D_\Lambda$, not in any deeper stratum, we
have from Lemma \ref{lem.divisors}, for $1\le i\le k$,
a $N_{G_{\bar\delta}}(\Lambda)$-equivariant morphism $\mathbb D_\Lambda\to
\bP(V/V_i)$.
This is surjective when $i=t$ and has image $\bP(V_{i+1}/V_i)$ for $i<t$.

For $\bar\delta\in \Delta/T$ the stabilizer $N_{G_{\bar\delta}}(\Lambda)$ acts on the fiber
$\vartheta_{[r]}^{-1}(\bar\delta)$ over $\bar\delta$ in $T'_{[r]}/T'$,
where $T'$ is as in \eqref{eqn.chainfirstT}.
We let $\mathcal{K}_j$ denote the set of $N_{G_{\bar\delta}}(\Lambda)$-orbits for this action.

\begin{prop}
\label{prop.mainformula}
Let $G$ act on $\Delta$,
and let $\bar\delta\in \Delta/T$.
Let us write
\[
[X\bar\delta \actsfromright G_{\bar\delta}/G_\delta]=A_{\bar{\delta}} +  B_{\bar{\delta}} 
\]
in
$$
\Burn_n(G_{\bar\delta}/G_\delta),
$$
where 
$A_{\bar{\delta}}$ records the contribution from $\mathbb X^{\circ}\cap X\bar\delta$, and 
$B_{\bar{\delta}}$, the contribution from strata obtained from exceptional divisors in the De Concini-Procesi model for $X$.
Then
$$
A_{\bar{\delta}}=
\sum_{[\sigma]\in \Sigma/G_{\bar\delta}}
(\Stab(\sigma)_\delta,
\Stab(\sigma)/\Stab(\sigma)_\delta
\actsfromleft k(T^\sigma \bar\delta),\rho_{\sigma}),
$$
where
\begin{itemize}
\item the sum is over $G_{\bar\delta}$-orbits of $\Sigma$; for each orbit an orbit representative $\sigma\in \Sigma$ is chosen,
\item 
the action of $\Stab(\sigma)$ on $\mathbb X^{\circ}\cap T^\sigma\bar\delta$ has constant stabilizer
$\Stab(\sigma)_\delta$,
\item
the representation
\[
\rho_{\sigma}\colon \Stab(\sigma)_\delta\to \GL(\mathfrak J/\mathfrak J^2)
\]
is defined, using the ideal
\[
\mathfrak J:=(x)_{x\in (\sigma^\vee\cap M)\setminus (\sigma^\perp\cap M)}(k[\sigma^\vee\cap M]/(x-\delta(x))_{x\in \sigma^\perp\cap M}k[\sigma^\vee\cap M]),
\]
\end{itemize}
and 
\begin{align*}
&B_{\bar{\delta}}=\\
&\ \sum_{\emptyset\ne I\subseteq \cI_{\bar\delta}} 
\sum_{[\Lambda]\in \cJ_I}
\sum_{[\bar\tau]\in \mathcal{K}_j}
\mathrm{ind}_{ N_{G_{\bar\delta}}(\Lambda)_{\bar\tau}/G_{\delta}}^{G_{\bar\delta}/G_\delta} \big( \psi_{\{ 1,\dots, t\}} \big(
[D_\Lambda^\circ\bar\tau\actsfromright
 N_{G_{\bar\delta}}(\Lambda)_{\bar\tau}/G_{\delta}
 %N_G(\Lambda)_{\bar\tau}/G_\delta
 ]^{\mathrm{naive}}_{(\cO(-1))} \big)
\big).
\end{align*}
\end{prop}

By analogy with \cite[Conv.\ 8.1]{KT-vector}, here $(\cO(-1))$ denotes the following collection of line bundles, indexed by $\{1,\dots,t\}$:
\begin{equation}
\label{eqn.convention}
\cO_{\bP(V_2/V_1)}(-1),\cO_{\bP(V_2/V_1)}(1)\otimes \cO_{\bP(V_3/V_2)}(-1),\dots.
\end{equation}

\begin{proof}
The contribution from $\mathbb X^{\circ}\cap X\bar\delta$ is obtained directly from the formula for the naive class in $\Burn_n(G_{\bar\delta}/G_\delta)$,
from Section \ref{sect:ebg}.
The term $B_{\bar\delta}$ is taken from the formula in Section \ref{sect:ebg}, where the additional sum over orbit representatives of $\mathcal{K}_j$ accounts for the components in \eqref{eqn.DLambdacomponents}.
\end{proof}

Let us fix a chain \eqref{eqn.chain}, which indexes a stratum $\mathbb D_\Lambda\subset \mathbb X_{\Sigma,\cL',[r]}$.
Let $\bar\delta\in \Delta/T$, and
with $T'$ as in \eqref{eqn.chainfirstT}, let
$\bar\tau\in \vartheta^{-1}_{[r]}(\bar\delta)$.
The group $N_{G_{\bar\delta}}(\Lambda)_{\bar\tau}$ acts on
\begin{equation}
\label{eqn.whatNGLambdaactson}
X'\bar\tau, \quad \bP(V_2/V_1), \quad \dots, \quad \bP(V_t/V_{t-1}), \quad \bP(V/V_t),
\end{equation}
and we have an $N_{G_{\bar\delta}}(\Lambda)_{\bar\tau}$-equivariant birational morphism from $D_\Lambda\bar\tau$ to the product of the varieties in \eqref{eqn.whatNGLambdaactson}.
Therefore:

\begin{lemm}
\label{lem.DLambda}
For $\bar\tau\in T'_{[r]}/T'$, mapping to $\bar\delta\in \Delta/T$, we have
\begin{align*}
[&D_\Lambda \bar\tau\actsfromright N_{G_{\bar\delta}}(\Lambda)_{\bar\tau}/G_\delta]_{(\cO(-1))}=\\
&\qquad[X'\bar\tau\times \bP(V_2/V_1)\times \dots\times \bP(V/V_t)\actsfromright N_{G_{\bar\delta}}(\Lambda)_{\bar\tau}/G_\delta]_{(\cO(-1))}.
\end{align*}
in $\Burn_{n,\{1,\dots,t\}}(N_{G_{\bar\delta}}(\Lambda)_{\bar\tau}/G_\delta)$.
\end{lemm}

For $\Lambda\ne\emptyset$, the expression on the right-hand side in Lemma \ref{lem.DLambda} may be taken as known, in the recursive determination of the classes \eqref{eqn.classdeltacomponent}.

\begin{theo}
\label{thm.main}
The class
\[
[X\bar\delta\actsfromright G_{\bar\delta}/G_\delta]=A_{\bar\delta}+B_{\bar\delta}
\]
in $\Burn_n(G_{\bar\delta}/G_\delta)$ may be computed by applying the formula for $A_{\bar\delta}$ from Proposition \ref{prop.mainformula} directly, and by computing
the classes
$$
[D^\circ_\Lambda\bar\tau\actsfromright N_{G_{\bar\delta}}(\Lambda)_{\bar\tau}/G_\delta]_{(\cO(-1))}^{\mathrm{naive}}\in 
\Burn_{n,\{1,\dots,t\}}(N_{G_{\bar\delta}}(\Lambda)_{\bar\tau}/G_\delta)
$$
appearing in the
formula for
$B_{\bar\delta}$ from Proposition \ref{prop.mainformula}
in a recursive fashion, starting with large $t=|I|$, using the formula
\begin{align*}
[&D^\circ_\Lambda\bar\tau\actsfromright N_{G_{\bar\delta}}(\Lambda)_{\bar\tau}/G_\delta]_{(\cO(-1))}^{\mathrm{naive}}=[D_\Lambda\bar\tau\actsfromright N_{G_{\bar\delta}}(\Lambda)_{\bar\tau}/G_\delta]_{(\cO(-1))} \\
&-\sum_{[\Lambda']}
\mathrm{ind}_{N_{G_{\bar\delta}}(\Lambda')_{\bar\tau}/G_\delta}^{N_{G_{\bar\delta}}(\Lambda)_{\bar\tau}/G_\delta}\big(\psi_{I',J}
\big(\tau_{I',J}[D^\circ_{\Lambda'}\bar\tau\actsfromright N_{G_{\bar\delta}}(\Lambda')_{\bar\tau}/G_\delta]_{(\cO(-1))}^{\mathrm{naive}}\big)\big)
\\
&-\sum_{[\Lambda'']}\sum_{[\bar\tau'']} \mathrm{ind}_{N_{G_{\bar\delta}}(\Lambda'')_{\bar\tau''}/G_\delta}^{N_{G_{\bar\delta}}(\Lambda)_{\bar\tau}/G_\delta}\big(\psi_{I'',J}
\big(\tau_{I'',J}[D^\circ_{\Lambda''}\bar\tau''\actsfromright N_{G_{\bar\delta}}(\Lambda'')_{\bar\tau''}/G_\delta]_{(\cO(-1))}^{\mathrm{naive}}\big)\big).
\end{align*}
The first sum is over $N_G(\Lambda)_{\bar\tau}$-conjugacy classes of chains
\[ \Lambda':\ \ \Gamma^1=\Gamma'^1\supset \Gamma'^2\supset\dots\supset \Gamma'^{t'} \]
strictly containing $\Lambda$ with the same largest member
$\Gamma^1=\Gamma'^1$; we put $I':=\{ 1,\ldots, t'\}$. 
The second sums are over
$N_G(\Lambda)_{\bar\tau}$-conjugacy classes of chains
\[ \Lambda'':\ \ \Gamma''^1\supset
\Gamma''^2\supset\dots\supset
\Gamma''^{t''} \]
containing $\Lambda$, with
$\Gamma''^1\supsetneq \Gamma^1$,
and $N_G(\Lambda'')_{\bar\tau}$-orbit representatives $\bar\tau''$ of the
fiber of
\[ T''_{[r]}/T''\to T'_{[r]}/T' \]
over $\bar\tau$, where $T''$ denotes $T_{\Gamma''^1}$;
we put $I'':=\{ 1, \ldots, t''\}$. 
In each sum, $J$ records the indices of the members of $\Lambda$ and is identified in an order-preserving fashion with $\{1,\dots,t\}$ to land in
$\Burn_{n,\{1,\dots,t\}}(N_{G_{\bar\delta}}(\Lambda)_{\bar\tau}/G_\delta)$.
\end{theo}

\begin{proof}
This follows from the formula in \S \ref{sect:ebg}, and the evident analogous formula for indexed equivariant Burnside groups.
By Lemma \ref{lem.divisors}, application of $\tau_{I',J}$ and $\tau_{I'',J}$ corrects the divisor characters in the first, respectively second sums in the formula.
\end{proof}

\section{Dimension 2}
\label{sect:2}

There is only one nontrivial action  on $\mathbb G_m$, namely $t\mapsto t^{-1}$, and it is linearizable. 

In dimension 2, it is known \cite[II.4.9, Exa.\ 7]{vosk} that every action on $\G_m^2$ factors through a subgroup of
\[ 
\mathfrak D_4\text{ or }C_2\times \fS_3\subset \GL_2(\Z).
\]
Subgroups of $\mathfrak D_4$ give rise to a regular action on $\bP^1\times \bP^1$.
Projecting from the identity of
$\G_m^2\subset \bP^1\times \bP^1$ gives an induced regular action on $\bP^2$, and this is linear.
So, we focus on
\[ G:=C_2\times \fS_3, \]
whose action we realize by
inverse and permutation on the coordinates of
\[ T\subset \G_m^3,\qquad T:=\{(t_1,t_2,t_3)\,|\,t_1t_2t_3=1\}. \]
This extends to a regular $G$-action on a del Pezzo surface $X$ of degree $6$; the action is obtained by regularizing the $\fS_3$ permutation action on the standard coordinates of 
$\bP^2$ together with the Cremona involution. 
As mentioned in the introduction:
\begin{itemize}
\item 
This action is not linearizable, by \cite{isk-s3};
there, the proof relied on the (equivariant) Minimal Model Program, specifically, on the classification of Sarkisov links. In \cite{HKTsmall} we provided an alternative proof, using partial information 
about the class of the action in the equivariant Burnside group.
\item 
There are no cohomological obstructions to stable linearizability,
$$
\rH^1(G', \Pic(X)) =0,
$$
for all subgroups $G'\subseteq G$. 
\item For every subgroup $G'\subsetneq G$, the
$G'$-action on $T$ is linearizable \cite[Section 9]{lemire}. 
\item The $G$-action is stably linearizable \cite[Prop. 9.11]{lemire}. 
\end{itemize}

%we compute the full class of the action.
%and use it to settle an open problem from 
%\cite[Remark 9.13]{lemire}.

We apply the procedure from Section~\ref{sect:class} to compute the class of the $G$-action on an equivariant compactification of $T$ in the Burnside group. The result is: 

\begin{prop}
\label{prop:class-dp6} 
The class in $\Burn_2(G)$ of the $G$-action on $X$ is: 
\begin{align*}
[X\actsfromright G]=(\mathrm{triv}&,G\actsfromleft k(X),())\\
+&(\fS_2,C_2\actsfromleft k(\bP^1),(1)) \\
+&(\text{diagonal in $C_2\times \fS_2$},C_2\actsfromleft k(\bP^1),(1))\\
+&{\color{red}(C_2,\fS_3\actsfromleft k(\bP^1),(1))}  \\
+&(C_2,\fS_2\actsfromleft k(\bP^1),(1))\\
+2&(C_2\times \fS_2,\mathrm{triv}\actsfromleft k,(e_1,e_2))\\
+2&(C_2\times \fS_2,\mathrm{triv}\actsfromleft k,(e_1+e_2,e_2))\\
+&(C_2\times C_3,\mathrm{triv}\actsfromleft k,((0,1),(1,1)))\\
+&
(C_3,\mathrm{triv}\actsfromleft k,(1,1)).
\end{align*}
\end{prop}

For comparison, we display the class of the linear action
$$
[\bP(1\oplus V_{\chi})\actsfromright G]\in \Burn_2(G),
$$
computed in 
\cite[Exa.~5.3]{KT-struct}; 
here $V_{\chi}$ is the standard 
2-dimensional representation of $\fS_3$, twisted by
a nontrivial character of the central $C_2$.  
This, in essence the only possible linear action, yields the class
\begin{align*}
& (\mathrm{triv},G\actsfromleft k(\bP(1\oplus V_{\chi})),())\\ 
%(\mathrm{triv},G\actsfromleft k(\bP^1)(t),\emptyset)
+&(\fS_2,C_2\actsfromleft k(\bP^1),(1))\\
+2&{\color{red} (C_2,\fS_3\actsfromleft k(\bP^1),(1))} \\
+2&
(C_2\times \fS_2,\mathrm{triv}\actsfromleft k,(e_1,e_2))\\
+2&(C_2\times\fS_2,\mathrm{triv}\actsfromleft k,(e_1+e_2,e_2))\\
+&(C_2\times C_3,\mathrm{triv}\actsfromleft k,((0,1),(1,1)))\\
+&(C_2\times C_3,\mathrm{triv}\actsfromleft k,((0,1),(1,2))).
\end{align*}
The term
\begin{equation}
\label{eqn:symb}
(C_2,\fS_3\actsfromleft k(\bP^1),(1))
\end{equation}
is {\em incompressible} (see Defn. 3.3 and Prop.~3.6 in  \cite{KT-vector}). It appears with different coefficients in the two expressions above.
By \cite[Prop.~3.4]{KT-vector}, the actions are not equivariantly birational. However, there are other terms in these formulas that distinguish the two actions, e.g., the terms with $C_2\times C_3$-stabilizer.

Recall that $X\times \bP^r$ and $\bP(1\oplus V_{\chi})\times \bP^r$, with 
trivial $G$-action on $\bP^r$, are equivariantly birational for $r\ge 2$, by \cite[Prop. 9.11]{lemire}. The case of $r=1$ is unknown. Computing  $\Burn_3(G)$ we find nothing to distinguish the classes 
$$
[X\times \bP^1\actsfromright G], [\bP(1\oplus V_{\chi})\times \bP^1\actsfromright G] \in \Burn_3(G).
$$ 
Indeed, all contributions from nontrivial stabilizers vanish in $\Burn_3(G)$. For terms with $C_2$- and $C_3$-stabilizers this follows immediately from Proposition~\ref{prop:s-vanishing}, while for terms with $C_2\times C_3$-stabilizer this requires further analysis.

\medskip

The rest of this section consists of a proof of Proposition ~\ref{prop:class-dp6}. It is an application of the algorithm from Section \ref{sect:class}, which we carry out in detail.
We have: 
\begin{itemize}
\item $T=\Delta$,
\item $\bar{\delta} = 1$, 
\item $G_{\bar{\delta}} = G$, $G_{\delta} = \mathrm{triv}$, 
\end{itemize}
and the formula from Proposition \ref{prop.mainformula} simplifies to
\[ [X\actsfromright G]=A+B. \]
We use the coordinates $t_1$ and $t_2$ to identify $T$ with $\G_m^2$, and recover the action of Example \ref{exam.notstabilizer}.
There is a corresponding basis $e_1$, $e_2$ of $N\cong \Z^2$.

\medskip
\noindent
{\em Step 1.} 
We compute $\cL'=\cL'(T)$, the lattice of distinguished subgroups of $G$, associated with subtori of $T$:
\[
\begin{array}{c|c}
\Gamma'&T_{\Gamma'} \\ \hline
G & \{(1,1,1)\} \\
\langle (0,(1,2))\rangle & \{(t,t,t^{-2})\} \\
\langle (0,(1,3))\rangle & \{(t,t^{-2},t)\} \\
\langle (0,(2,3))\rangle & \{(t^{-2},t,t)\} \\
\langle (1,(1,2))\rangle & \{(t,t^{-1},1)\} \\
\langle (1,(1,3))\rangle & \{(t,1,t^{-1})\} \\
\langle (1,(2,3))\rangle & \{(1,t,t^{-1})\} \\
\mathrm{triv} & T
\end{array}
\]

\medskip
\noindent
{\em Step 2}: We construct a smooth projective $G$-invariant fan, with respect to which property $(\mathrm{E})$ holds for every $T_{\Gamma'}$; this has rays generated by
\[ (1,0), (1,1), (0,1), (-1,0), (-1,-1), (0,-1). \]

\medskip
\noindent
{\em Step 3}: We subdivide to obtain a fan $\Sigma$ satisfying the additional property, that no pair of rays in a single $G$-orbit spans a cone of $\Sigma$;
the ray generators are
\begin{align*}
&(1,0), (2,1), (1,1), (1,2), (0,1), (-1,1), (-1,0), (-2,-1), \\
&\qquad\qquad\qquad\qquad\qquad\qquad (-1,-1), (-1,-2), (0,-1), (1,-1).
\end{align*}

\medskip
\noindent
{\em Step 4.} We find a positive integer $r$, such that the stabilizer locus in $T$ is in the union of the $r$-torsion translates of subtori in $\cG_{\cL'}$, and the same holds for the $\Stab(\sigma)$-action on $T^\sigma$, for all $\sigma\in \Sigma$.
In $T$, the stabilizer locus consists of the one-dimensional subtori above, together with
\[
(1,1,1),
(-1,-1,1),(-1,1,-1),(1,-1,-1),
(\zeta,\zeta,\zeta),(\zeta^2,\zeta^2,\zeta^2).
\]
For $\sigma\in \Sigma(1)$ we have $\Stab(\sigma)\cong \Z/2\Z$ acting on $T^\sigma\cong \G_m$, fixing $\pm 1$.
So we take
\[ r=6. \]

\medskip
\noindent
{\em Step 5.}
We carry out the De Concini-Procesi blow-up procedure, which in this case amounts to blowing up the $6$-torsion of $T$ in $X=X_\Sigma$ to obtain
\[ X_{\Sigma,\cL',[6]}\cong B\ell_{\text{$36$ points}}X. \]

\medskip
\noindent
{\em Step 6.}
We compute $A$ directly as
\[ A=(\mathrm{triv},G\actsfromleft k(X),()). \]
This is the contribution from the zero cone in the formula from $\bm{\Sigma}_1$ in Proposition \ref{prop.mainformula}.
The two orbits of $1$-dimensional cones $\sigma$ lead to action of $\Stab(\sigma)\cong \Z/2\Z$ on $T^\sigma\cong \G_m$ with trivial generic stabilizer, hence no contribution to the equivariant Burnside group.
As well there is no contribution from the $2$-dimensional cones, which form a single orbit with trivial stabilizer.

\medskip
\noindent
{\em Step 7.}
We compute $B$ by the procedure of Theorem \ref{thm.main}.
We have $\ell=3$:
$$
\mathbb D = \mathbb D_1\cup \mathbb D_2\cup \mathbb D_3  \subset X_{\Sigma, \cL', [6]},
$$
with respective conjugacy class representatives of $\cL'\setminus\{\mathrm{triv}\}$ from the table in Step 1:
\[ \Gamma'_1=G,\qquad \Gamma'_2=\langle(0,(1,2))\rangle,\qquad \Gamma'_3=\langle(1,(1,2))\rangle. \]
Together, $\Gamma'_2$ and $\Gamma'_3$ generate $C_2\times \fS_2\cong \mathfrak K_4$.

%the three $G$-invariant divisors correspond to the conjugacy class representatives $\Lambda_i$, $i=1,\ldots, 3$, of elements of $\mathcal L'\setminus \{ \ker(\nu)\}$ :
%
%For $i=2,3$, we put
%$$
%N_{i,\bar{\tau}}:=N_{G}(\Lambda_i)_{\bar{\tau}}. 
%$$
With $\mathcal I=\{ 1, 2, 3\}$
we have nonempty $D_I$ 
corresponding to the following subsets $I\subseteq \mathcal I$:
$$
\{1\}, \quad \{2\},\quad \{3\},\quad \{ 1, 2\}, \quad \{ 1, 3\}. 
$$
For each $I$, there is
exactly one conjugacy class of chains, with
representative
\[ \{\Gamma'_i\,|\,i\in I\}. \]
We list representative chains $\Lambda$, with corresponding $N_G(\Lambda)$ and $N_G(\Lambda)$-orbits of $T'_{[6]}/T'$:
\[
\begin{array}{c|c|c|c|c}
\Lambda&N_G(\Lambda)&T'&T'_{[6]}/T'&\text{$N_G(\Lambda)$-orbits} \\ \hline
G & G & \{1\} & T[6] & \text{see Table \ref{orbits}} \\
\langle (0,(1,2))\rangle & \mathfrak K_4 & \{(t,t,t^{-2})\} & \text{$\mu_6$ (via $t_1^{-1}t_2$)} & \text{see below} \\
\langle (1,(1,2))\rangle & \mathfrak K_4 & \{(t,t^{-1},1)\} & \text{$\mu_6$ (via $t_1t_2$)} & \text{see below} \\
G \supset \langle (0,(1,2))\rangle & \mathfrak K_4 & \{1\} & T[6] & \text{see Table \ref{orbits}} \\
G \supset \langle (1,(1,2))\rangle & \mathfrak K_4 & \{1\} & T[6] & \text{see Table \ref{orbits}}
\end{array}
\]
When $T'$ has dimension $1$, we identify $T'_{[6]}/T'$ with $\mu_6$ by the indicated coordinate function.
The action by $N_G(\Lambda)=\mathfrak K_4$ has orbits
\[ \{1\},\quad \{-1\},\quad \{\zeta,\zeta^2\},\quad \{-\zeta,-\zeta^2\}. \]
When $\Lambda=\{\langle(0,(1,2))\rangle\}$, the elements in orbits of size $2$ have stabilizer $\langle(1,(1,2))\rangle$.
When $\Lambda=\{\langle(1,(1,2))\rangle\}$, the elements in orbits of size $2$ have stabilizer $\langle(0,(1,2))\rangle$.

\begin{table}
\[
\begin{array}{l|l}
\text{stabilizer}&\text{orbit [$G$-stabilizer]} \\ \hline
\mathfrak K_4&(1,1,1)\,[G]\\
&(-1,-1,1)\,[\mathfrak K_4]\\ \hline
\langle(0,(1,2))\rangle&(\zeta,\zeta,\zeta)\,(\zeta^2,\zeta^2,\zeta^2)\,[\fS_3] \\
&(-\zeta,-\zeta,\zeta)\,(-\zeta^2,-\zeta^2,\zeta^2)\, [\langle(0,(1,2))\rangle] \\
\langle(1,(1,2))\rangle &(\zeta,\zeta^2,1)\,(\zeta^2,\zeta,1)\,[\langle(1,(1,2))\rangle] \\
&(-\zeta,-\zeta^2,1)\,(-\zeta^2,-\zeta,1)\,[\langle(1,(1,2))\rangle] \\
\langle(1,\mathrm{id})\rangle & (1,-1,-1)\,(-1,1,-1) \\ \hline
\mathrm{triv} & (\zeta,-\zeta,-\zeta)\,(\zeta^2,-\zeta^2,-\zeta^2)\,(-\zeta,\zeta,-\zeta)\,(-\zeta^2,\zeta^2,-\zeta^2) \\
& (1,\zeta,\zeta^2)\,(1,\zeta^2,\zeta)\,(\zeta,1,\zeta^2)\,(\zeta^2,1,\zeta) \\
&(1,-\zeta,-\zeta^2)\,(1,-\zeta^2,-\zeta)\,(-\zeta,1,-\zeta^2)\,(-\zeta^2,1,-\zeta) \\
&(\zeta,-\zeta^2,-1)\,(\zeta^2,-\zeta,-1)\,(-\zeta^2,\zeta,-1)\,(-\zeta,\zeta^2,-1)\,[\mathrm{triv}] \\
&(-1,\zeta,-\zeta^2)\,(-1,\zeta^2,-\zeta)\,(\zeta,-1,-\zeta^2)\,(\zeta^2,-1,-\zeta) \\
&(-1,-\zeta,\zeta^2)\,(-1,-\zeta^2,\zeta)\,(-\zeta,-1,\zeta^2)\,(-\zeta^2,-1,\zeta)
\end{array}
\]
\caption{Orbits of $T[6]$ under $\mathfrak K_4\subset G$.
Orbits under $G$ are unions of $\mathfrak K_4$-orbits; for each a representative is identified, with $G$-stabilizer displayed in brackets $[\,]$.}
\label{orbits}
\end{table}

As indicated in Theorem \ref{thm.main}, we start the computation of $B$ by looking at contributions with $t=2$; for these, we have
$\mathbb D^\circ_\Lambda=\mathbb D_\Lambda$.

\begin{itemize}
\item 
$\Lambda=\{G\supset \langle(0,(1,2))\rangle\}$:
We have $N_G(\Lambda)=\mathfrak K_4$, with
\[ V_1=0\qquad\text{and}\qquad V_2=k\cdot(1,1). \]
Following Lemma \ref{lem.DLambda}, we have
\[
[D_\Lambda\bar\tau\actsfromright (\mathfrak K_4)_{\bar\tau}]_{(\cO(-1))}=
[\{1\}\bar\tau\times \bP(V_2/V_1)\times \bP(V/V_2)\actsfromright (\mathfrak K_4)_{\bar\tau}]_{(\cO(-1))},
\]
as a point, with pair of characters $e_1$ determined by $V_2/V_1$, and $e_1+e_2$ determined by $V/V_2$.
So, \eqref{eqn.convention} gives $e_1$ and $e_2$ as characters of
$(\cO(-1))$.
Applying $\psi_{\{1,2\}}$ to get an element of $\Burn_2((\mathfrak K_4)_{\bar\tau})$, we only get something nontrivial when $(\mathfrak K_4)_{\bar\tau}={\mathfrak K_4}$.
There are two contributions from Table \ref{orbits};
when we apply induction to $\Burn_2(G)$ we obtain
\begin{align*}
\psi_{\{1,2\}}\big(&[D_\Lambda\bar\tau\actsfromright (\mathfrak K_4)_{\bar\tau}]_{(\cO(-1))}\big)=\\
&\begin{cases} (C_2\times\fS_2,\mathrm{triv}\actsfromleft k,(e_1,e_2)), & \text{if $(\mathfrak K_4)_{\bar\tau}={\mathfrak K_4}$},\\ 0, & \text{otherwise}.\end{cases}
\end{align*}
\item $\Lambda=\{G\supset \langle(1,(1,2))\rangle\}$: The computation is similar, with 
$$
V_2=k\cdot(1,-1),
$$
and we obtain
\begin{align*}
\psi_{\{1,2\}}\big(&[D_\Lambda\bar\tau\actsfromright (\mathfrak K_4)_{\bar\tau}]_{(\cO(-1))}\big)=\\
&\begin{cases} (C_2\times\fS_2,\mathrm{triv}\actsfromleft k,(e_1+e_2,e_2)), & \text{if $(\mathfrak K_4)_{\bar\tau}={\mathfrak K_4}$},\\ 0, & \text{otherwise}.\end{cases}
\end{align*}
\end{itemize}

We proceed to cases with $t=1$.

\begin{itemize}
\item $\Lambda=\{G\}$: We have
$V_1=0$.
So,
\begin{align*}
\psi_{\{1\}}\big(&[D^\circ_\Lambda\bar\tau\actsfromright G_{\bar\tau}]_{(\cO(-1))}\big) \\
=
&\begin{cases}
\psi_{\{1\}}\big([D_\Lambda\bar\tau\actsfromright G_{\bar\tau}]_{(\cO(-1))}\big)-\rC_0-\rC_1,
&\text{if $G_{\bar\tau}\supseteq \mathfrak K_4$},\\
\psi_{\{1\}}\big([D_\Lambda\bar\tau\actsfromright G_{\bar\tau}]_{(\cO(-1))}\big),&\text{if $|G_{\bar\tau}|\in\{1,2,6\}$},
\end{cases}
\end{align*}
where for $i\in \{0,1\}$,
\[
\rC_i:=\mathrm{ind}_{\mathfrak K_4}^{G_{\bar\tau}}\big(\psi_{\{1,2\}}\big([D_{\{G\supset \langle(i,(1,2))\rangle \}}\bar\tau\actsfromright \mathfrak K_4]_{(\cO(-1))}\big)\big);
\]
by Lemma \ref{lem.DLambda}, we have
\[
[D_\Lambda\bar\tau\actsfromright G_{\bar\tau}]_{(\cO(-1))}=
[\{1\}\bar\tau\times \bP(V)\actsfromright G_{\bar\tau}]_{(\cO(-1))}.
\]
The nontrivial contributions come from three values of $\bar\tau$.
When $\bar\tau=(1,1,1)$, we get
\begin{align*}
\psi_{\{1\}}\big(&[D^\circ_\Lambda\bar\tau\actsfromright G_{\bar\tau}]_{(\cO(-1))}\big)=
(C_2,\fS_3\actsfromleft k(\bP^1),(1))\\
&\qquad+(C_2\times C_3,\mathrm{triv}\actsfromleft k,((0,1),(1,1))).
\end{align*}
The cases $\bar\tau=(-1,-1,1)$ and $\bar\tau=(\zeta,\zeta,\zeta)$ give
\[
(C_2,\fS_2\actsfromleft k(\bP^1),(1)),\quad\text{respectively}\quad(C_3,\mathrm{triv}\actsfromleft k,(1,1)).
\]
\item $\Lambda=\{\langle (0,(1,2))\rangle \}$:
The only nontrivial contribution is from
$\bar\tau=1$.
Then,
\begin{align}
\begin{split}
\label{eqn.withcoefficienttwo}
\psi_1\big([&D_\Lambda\bar\tau\actsfromright \mathfrak K_4]_{(\cO(-1))}\big)
=(\fS_2,C_2\actsfromleft k(\bP^1),(1)) \\
&\qquad\qquad+2(\mathfrak K_4,\mathrm{triv}\actsfromleft k,(e_1,e_1+e_2))\in \Burn_2(\mathfrak K_4).
\end{split}
\end{align}
We get $\psi_1\big([D^\circ_\Lambda\bar\tau\actsfromright \mathfrak K_4]^{\mathrm{naive}}_{(\cO(-1))}\big)$,
according to Theorem \ref{thm.main}, by subtracting contributions from
\[ \Lambda''=\{G,\langle 0,(1,2)\rangle\},\quad \bar\tau''\in\{(1,1,1),(-1,-1,1),(\zeta,\zeta,\zeta),(-\zeta,-\zeta,\zeta)\}. \]
When $\bar\tau''\in\{(1,1,1),(-1,-1,1)\}$, we apply $\tau_{\{1,2\},\{2\}}$ to the
indexed equivariant Burnside group element
\[ (\mathfrak K_4\subseteq \mathfrak K_4,\mathrm{triv}\actsfromleft k,(),(e_1,e_2))
\in \Burn_{2,\{1,2\}}(\mathfrak K_4) \]
to yield, in each case, weights $e_1$ and $e_1+e_2$,
thereby cancelling the term in \eqref{eqn.withcoefficienttwo} with coefficient $2$.
When $\bar\tau''\in\{(\zeta,\zeta,\zeta),(-\zeta,-\zeta,\zeta)\}$ the
contribution is trivial in $\Burn_2(\mathfrak K_4)$.
So
\[
\psi_1\big([D^\circ_\Lambda\bar\tau\actsfromright \mathfrak K_4]^{\mathrm{naive}}_{(\cO(-1))}\big)
=(\fS_2,C_2\actsfromleft k(\bP^1),(1)).
\]
\item $\Lambda=\{\langle (1,(1,2))\rangle \}$: The computation is similar.
There is only a nontrivial computation for $\bar\tau=1$, and we obtain
\[
\psi_1\big([D^\circ_\Lambda\bar\tau\actsfromright \mathfrak K_4]^{\mathrm{naive}}_{(\cO(-1))}\big)
=(\text{diagonal in $C_2\times \fS_2$},C_2\actsfromleft k(\bP^1),(1)).
\]
\end{itemize}

Combining the contributions, we obtain the formula in the statement of Proposition~\ref{prop:class-dp6}.

\section{Dimension 3}
\label{sect:dim3}

In this section, we analyze 3-dimensional tori, following \cite{kun}.
This is the smallest dimension where cohomology can obstruct rationality and linearizability. 
We have two motivating problems:
\begin{itemize}
\item Find nonlinearizable actions with vanishing $\rH^1(G',\Pic(X))$, for all $G'\subseteq G$, this is the analogue of \cite{isk-s3}.
\item Investigate the relation between 
$$
[X\actsfromright G]\quad \text{ and  } \quad \rH^1(G',\Pic(X)),
$$
where $X$ is a smooth $G$-equivariant projective compactification of $T$ and $G'\subseteq G$. 
\end{itemize}

Any action on a torus $T=\mathbb G_m^3$ factors through a subgroup
of 
$$
C_2\times \fS_3\times C_2\quad \text{ or } \quad C_2\times \fS_4,
$$
and the second group admits 3 different actions, labeled C, S, and P in \cite{kun}. The first group is realized on a product of a del Pezzo surface of degree 6 with $\bP^1$, with the natural action 
of $G':=C_2\times \fS_3$ on the DP6 (described in Section~\ref{sect:2}) and $C_2$ on $\bP^1$. 
As mentioned in Section \ref{sect:2} and \cite[Rem. 9.13]{lemire}, it is already an open problem, whether the $G'$-action is linearizable.
%By Proposition~\ref{prop:stable-non}, 

The other actions are realized as follows \cite[Section 2]{kun}:
\begin{itemize}
\item[(C)] on $\bP^1\times \bP^1\times\bP^1$,
\item[(S)] on the blowup of $\bP^3$ in the four coordinate points and the six lines through these points; here $\fS_4$ permutes the coordinates in $\bP^3$ and $C_2$ is the Cremona involution;
\item[(P)] on the (singular) hypersurface
$$
\{ x_1x_2x_3x_4=y_1y_2y_3y_4 \} \subset (\bP^1)^4,
$$
where $\fS_4$ acts by permuting  the factors and $C_2$ switches $x_i$ and $y_i$, for all $i$.  
\end{itemize}

The $G=C_2\times \fS_4$-action in type (C) is linearizable. The following proposition covers the types (S) and (P), see \cite[Fig. 4]{kun}.

\begin{prop}
\label{prop:kun}
Assume that $T$ admits the action of the Klein group $G:=\bZ/2\oplus \bZ/2$ such that  
$G\subset\GL(M)$ is generated by
$$
\begin{pmatrix}
-1 & 0 & 0 \\
0 & 0& -1 \\
0& -1& 0
\end{pmatrix}, \quad 
\begin{pmatrix}
-1 & -1 & -1 \\
0 & 0& 1 \\
0& 1& 0
\end{pmatrix}. 
$$
%$$
%\sigma_{3}:=\sigma_1\cdot \sigma_2 = 
%\begin{pmatrix}
%1 & 1 & 1 \\
%0 & -1 & 0 \\
%0  & 0 & -1 
%\end{pmatrix}
%$$
Let $X$ be a smooth projective $G$-equivariant compactification of $T$. 
Then
\begin{enumerate}
\item $\rH^1(G, \Pic(X)) = \bZ/2$,
and the action is not stably linearizable, 
\item $[X\actsfromright G]=(\mathrm{triv},G\actsfromleft k(X),())$ in $\Burn_3(G)$.
\end{enumerate}
\end{prop}

\begin{proof}
The first statement is one of the key results in \cite[Section 4]{kun}.
The second follows by an argument as in the proof of Proposition~\ref{prop:s-vanishing}: any 
symbol with a nontrivial stabilizer that arises, vanishes in $\Burn_3(G)$. 
\end{proof}

\bibliographystyle{plain}
\bibliography{burntoric}

\end{document}